\documentclass[10pt]{article}

\usepackage{amssymb,amsthm,amsmath,hyperref}

\usepackage{graphicx,float}

\numberwithin{equation}{section}

\newtheorem{thm}{Theorem}[section]
\newtheorem{lem}{Lemma}[section]

\newtheorem{prop}{Proposition}[section]
\theoremstyle{definition}
\newtheorem{defn}{Definition}[section]
\theoremstyle{remark}
\newtheorem{rem}{Remark}[section]

\def\supetage#1#2{
\sup_{\scriptstyle {#1}\atop\scriptstyle {#2}} }

\allowdisplaybreaks

\begin{document}
\title{Compressible flows with a density-dependent viscosity coefficient\thanks{
This work is supported by NSFC 10871175} }
\author{Ting Zhang\thanks{E-mail: zhangting79@hotmail.com}, Daoyuan Fang\thanks{E-mail:
dyf@zju.edu.cn}\\
\textit{\small Department of Mathematics, Zhejiang University,
Hangzhou 310027, China} }
\date{}
\maketitle
\begin{abstract}
  We prove the global existence of weak solutions for the 2-D compressible
  Navier-Stokes equations with a density-dependent viscosity
  coefficient ($\lambda=\lambda(\rho)$). Initial data and solutions are small in energy-norm
  with nonnegative densities having arbitrarily large sup-norm.
  Then, we show that if there is a vacuum domain at the
  initial time, then the vacuum domain will retain for all time, and
  vanishes as time goes to infinity.  At last, we show that the condition of $\mu=$constant will induce
   a singularity of the system at vacuum. Thus, the viscosity coefficient $\mu$ plays a key role in the
Navier-Stokes equations.
\end{abstract}
\section{Introduction}
In this paper, we consider the following 2-D compressible
Navier-Stokes equations
    \begin{equation}
    \left\{
    \begin{array}{l}
        \rho_t+\mathrm{div}(\rho u)=0,\\
            (\rho u)_t+\mathrm{div}(\rho u\otimes u)+\nabla P=
            \mu\Delta
            u+\nabla((\mu+\lambda(\rho))\mathrm{div}u)+\rho f,
    \end{array}
    \right.\label{2VDD-E1.1}
    \end{equation}
for $x\in\mathbb{R}^2$ and $t>0$, with the boundary and initial
conditions
    \begin{equation}
      u(x,t)\rightarrow0,
      \ \rho(x,t)\rightarrow \tilde{\rho}>0, \ \textrm{ as
      }|x|\rightarrow\infty, \ t>0,
    \end{equation}
    \begin{equation}
      (\rho,u)|_{t=0}=(\rho_0,u_0).\label{2VDD-E1.2}
    \end{equation}
Here $\rho({x},t)$, $u({x},t)$ and $P=P(\rho)$ stand for the fluid
density, velocity and pressure respectively, $f$ is a given external
force, the dynamic viscosity coefficient $\mu$ is a positive
constant, the second viscosity coefficient $\lambda=\lambda(\rho)$
is a function of $\rho$.

At first, we prove the global existence of weak solutions that are
in an "intermediate" regularity class in which energies are small,
but oscillations are arbitrarily large. Specifically, we fix a
positive constant  $\tilde{\rho}$, assume that
$(\rho_0-\tilde{\rho},u_0)$ are small  in $L^2$,  and
$(\rho_0,u_0)\in L^\infty\times H^1$ with no restrictions on their
norms. Our existence result accommodates a wide class of pressures
$P$, including pressures that are not monotone in $\rho$. Since the
solutions may exhibit vacuum states and discontinuities in density
and velocity gradient across hypersurfaces, our results are
consequently much less regular and much more general than the
well-known small-smooth theory, such as
\cite{Danchin2004,Marsumura83}. This existence result generalizes
and improves upon the earlier result of
    Vaigant-Kazhikhov \cite{Vaigant} in two significant ways: the space domain is unbounded and the
    initial density may vanish in an
open set. It also generalizes and improves upon earlier results of
Hoff-Santos \cite{Hoff2008} and Hoff
\cite{Hoff2005,Hoff1995,Hoff1995-2} in two significant ways: the
second viscosity
    coefficient $\lambda$ is a function of the density $\rho$, and
     we omit the condition $\int (1+|x|^2)^a(\rho_0|u_0|^2+G(\rho_0))dx\leq
     M_0$ with a constant $a>0$.

We now give a precise formulation of our existence result.
\begin{defn}
We say that $(\rho,u)$ is a weak solution of
(\ref{2VDD-E1.1})--(\ref{2VDD-E1.2}), if $\rho$ and $u$ are suitably
integrable and satisfy that
\begin{itemize}
    \item
        \begin{equation}
          \left.\int\rho\phi
          dx\right|^{t_2}_{t_1}=\int^{t_2}_{t_1}\int(\rho\phi_t+\rho
          u\cdot\nabla\phi)dxdt\label{2VDD-E1.3}
        \end{equation}
        for all times $t_2\geq t_1\geq0$ and all $\phi\in
        C^1_0(\mathbb{R}^2\times[t_1,t_2])$,
    \item
      \begin{eqnarray}
          &&\left.\int\rho u\psi
          dx\right|^{t_2}_{t_1}-\int^{t_2}_{t_1}\int\{\rho
          u\cdot\psi_t+\rho(u\cdot\nabla \psi
          )\cdot u+P\mathrm{div}\psi\}dxdt\nonumber\\
                &=&-\int^{t_2}_{t_1}\int
          \{\mu \partial_k
          u^j\partial_k\psi^j+(\mu+\lambda)\mathrm{div}u\mathrm{div}\psi
          -\rho f\psi\}dxdt\label{2VDD-E1.4}
        \end{eqnarray}
        for all times $t_2\geq t_1\geq0$ and all $\psi=(\psi^1,\psi^2)\in
        (C^1_0(\mathbb{R}^2\times[t_1,t_2]))^2$.
\end{itemize}
\end{defn}

Concerning the pressure $P$, viscosity coefficients $\mu$ and
$\lambda$, we fix $0<\tilde{\rho}<\bar{\rho}$ and assume that
    \begin{equation}
      \left\{
      \begin{array}{l}
      P\in C^1([0,\bar{\rho}]),\ \lambda\in C^2([0,\bar{\rho}]),\\
       \mu>0,
      \ \lambda(\rho)\geq0,\ \rho\in[0,\bar{\rho}],\\
           P(0)=0,
              \ P'(\tilde{\rho})>0, \\
              (\rho-\tilde{\rho})[P(\rho)-P(\tilde{\rho})]>0,
                \ \rho\neq\tilde{\rho},\ \rho\in[0,\bar{\rho}],\\
              P\in C^2([0,\bar{\rho}])
              \ \textrm{ or }\    \frac{P(\cdot)}{2\mu+\lambda(\cdot)}
              \ \textrm{ is a monotone function on}\ [0,\bar{\rho}].
      \end{array}
      \right.\label{2VDD-E1.5}
    \end{equation}
Let $G$ be the potential energy density, defined by
    \begin{equation}
      G(\rho)=\rho\int^{\rho}_{\tilde{\rho}}\frac{P(s)-P(\tilde{\rho})}{s^2}ds.
    \end{equation}
Then for any $g\in C^2([0,\bar{\rho}])$ with
$g(\tilde{\rho})=g'(\tilde{\rho})=0$, there is a constant $C$ such
that $|g(\rho)|\leq CG(\rho)$, $\rho\in[0,\bar{\rho}]$.

    We measure the sizes of the initial data and the external
    force by
        \begin{equation}
          C_0=\int\left(\frac{1}{2}\rho_0|u_0|^2+G(\rho_0)\right)dx,
        \end{equation}
            \begin{equation}
              C_f=\sup_{t\geq0}\|f(\cdot,t)\|_{L^2}^2+
              \int^\infty_0\left(
                \|f(\cdot,t)\|_{L^2}+\|f(\cdot,t)\|_{L^2}^2
                +\sigma^4\|\nabla f\|_{L^4}^4+\sigma^3\|
                f_t(\cdot,t)\|_{L^2}^2
              \right)dt,
            \end{equation}
 and
        \begin{eqnarray}
          M_q&=&\int^\infty_0\left(
      \sigma^2\| f_t\|_{L^2}^2+\sigma^3\|\nabla
      f\|_{L^4}^4+\sigma^{2+\frac{q}{2}}\|f_t\|_{L^{2+q}}^{2+q}
      +\sigma^{2+\frac{q}{2}}\|\nabla f\|_{L^{4+2q}}^{4+2q}
      \right)dt\nonumber
          \\                    &&+\int|\nabla u_0|^2dx+\sup_{t\geq0}\|f\|_{L^{2+q}},\label{2VDD-E1.10}
        \end{eqnarray}
     where  $\sigma(t)=\min\{1,t\}$ and $q$ is
     a constant satisfying
        \begin{equation*}
          q\in(0,2)\ \textrm{ and }
          q^2<\frac{4\mu}{\mu+\lambda(\rho)},
          \ \forall\ \rho\in[0,\bar{\rho}].
        \end{equation*}

As in \cite{Hoff2008,Hoff2005}, we recall the definition of the
vorticity matrix $w^{j,k}=\partial_k u^j-\partial_j u^k$, and
define the function
    \begin{equation}
      F=(\lambda+2\mu)\mathrm{div}u-P(\rho)+P(\tilde{\rho}).
    \end{equation}
We also define the convective derivative $\frac{D}{Dt}$ by
$\frac{Dw}{Dt}=\dot{w}=w_t+u\cdot\nabla w$,  the H\"{o}lder norm
        $$
        <v>^\alpha_A=\supetage{x,y\in A}{x\neq
        y}\frac{|v(x)-v(y)|}{|x-y|^\alpha},
        $$
and
        $$
        <g>^{\alpha,\beta}_{A\times[t_1,t_2]}=\supetage{(x,t),(y,s)\in A\times[t_1,t_2]}{(x,t)\neq
        (y,s)}\frac{|g(x,t)-g(y,s)|}{|x-y|^\alpha+|t-s|^\beta},
        $$
where  $v:A\subseteq \mathbb{R}^2\rightarrow \mathbb{R}^2$,
 $g:A\times[t_1,t_2]\rightarrow \mathbb{R}^2$ and
 $\alpha,\beta\in(0,1]$.

The following is the global existence result of this paper.

\begin{thm}\label{2VDD-T1.1}
  Assume that conditions (\ref{2VDD-E1.5})--(\ref{2VDD-E1.10}) hold. Then,
  for a
  given  positive number $M$ (not necessarily small) and
  $\bar{\rho}_1\in(\tilde{\rho},\bar{\rho})$, there are positive
  numbers $\varepsilon$, $C$
  and  $\theta$, such that,  the Cauchy problem
  (\ref{2VDD-E1.1})--(\ref{2VDD-E1.2}) with
  the  initial data $(\rho_0,u_0)$ and external force $f$
  satisfying
    \begin{equation}
      \left\{
      \begin{array}{l}
       0\leq \inf \rho_0\leq\sup\rho_0\leq\bar{\rho}_1,\\
            C_0+C_f\leq\varepsilon,\\
                M_q\leq M,
      \end{array}
      \right.
    \end{equation}
 has a
global weak solution $(\rho,u)$ in the sense of
(\ref{2VDD-E1.3})--(\ref{2VDD-E1.4}) satisfying
    \begin{equation}
       C^{-1}\inf\rho_0\leq\rho\leq\bar{\rho},\
      \textrm{a.e.}\label{2VDD-E1.15}
    \end{equation}
        \begin{equation}
          \left\{
          \begin{array}{l}
            \rho-\tilde{\rho},\ \rho u\in C([0,\infty);H^{-1}),\\
                \nabla u\in L^2(\mathbb{R}^2\times[0,\infty)),
          \end{array}
          \right.
        \end{equation}
        \begin{equation}
          <u>^{\alpha,\frac{\alpha}{2+2\alpha}}_{\mathbb{R}^2\times[\tau,\infty)}+
          \sup_{t\geq \tau}\left(\|\nabla F(\cdot,t)\|_{L^2}+\|\nabla
          w(\cdot,t)\|_{L^2}\right)
          \leq C(\alpha,\tau)(C_0+C_f)^\theta,\label{2VDD-E1.15-1}
        \end{equation}
         \begin{equation}
          \sup_{t>0}\int|\nabla u|^2dx+\int^\infty_0\int\sigma
          |\nabla\dot{u}|^2dxds\leq C,\label{2VDD-E1.15-2}
        \end{equation}
    \begin{equation}
      \int^T_0\|F(\cdot,t)\|_{L^\infty}dt\leq C(T),
    \end{equation}
    \begin{equation}
          <F>^{{\alpha'},\frac{{\alpha'}}{2+2{\alpha'}}}_{\mathbb{R}^2\times[\tau,T]}+
     <w>^{{\alpha'},\frac{{\alpha'}}{2+2{\alpha'}}}_{\mathbb{R}^2\times[\tau,T]}\leq
    C(\tau,T),\label{2VDD-E1.22}
    \end{equation}
where $\alpha\in(0,1)$, $\tau >0$ and
${\alpha'}\in(0,\frac{q}{2+q}]$,  and
    \begin{eqnarray}
      &&\sup_{t>0}\int\left(
      \frac{1}{2}\rho|u|^2+|\rho-\tilde{\rho}|^2
      +\sigma|\nabla u|^2      \right)dx\nonumber\\
            &&+\int^\infty_0\int\left(
            |\nabla u|^2+\sigma|(\rho u)_t+\mathrm{div}(\rho u\otimes
            u)|^2+\sigma^2|\nabla\dot{u}|^2     \right)dxdt\nonumber\\
      &\leq& C(C_0+C_f)^\theta.\label{2VDD-E1.17}
    \end{eqnarray}
In addition, in the case that $\inf\rho_0>0$, the term
$\int^\infty_0\int\sigma|\dot{u}|^2dxdt$ may be included on the left
hand side of (\ref{2VDD-E1.17}).
\end{thm}

\begin{rem}
 Here, $\theta$ is a universal positive constant (we choose $\theta=\frac{1}{2}$ in this paper),
   $\varepsilon$ and $C$ depend on  $\tilde{\rho}$,
  $\bar{\rho}_1$, $\bar{\rho}$, $P$, $\lambda$, $\mu$, $q$ and
  $M$.
\end{rem}

\begin{rem}
  For example, we can choose that $P=A\rho^\gamma$ and
  $\lambda(\rho)=c\rho^\beta$ with $\gamma\geq 1$ and
  $\beta\geq2$, where $A$ and $c$ are two positive constants.
  Also, we can choose that $\lambda$ is a positive constant and $P=A\rho^\gamma$  with $\gamma\geq
  1$.
\end{rem}

\begin{rem}
Considering the non-vacuum case, i.e., $\rho_0\geq
2\underline{\rho}>0$, we can replace the condition (\ref{2VDD-E1.5})
by
   \begin{equation}
      \left\{
      \begin{array}{l}
      P\in C^1([0,\bar{\rho}]),\ \lambda\in C^2([\underline{\rho},\bar{\rho}]),\\
       \mu>0,
      \ \lambda(\rho)\geq0,\ \rho\in[0,\bar{\rho}]\\
           P(0)=0,\   P'(\tilde{\rho})>0,\\
           (\rho-\tilde{\rho})[P(\rho)-P(\tilde{\rho})]>0,
                \ \rho\neq\tilde{\rho},\ \rho\in[\underline{\rho},\bar{\rho}],\\
             P\in C^2([\underline{\rho},\bar{\rho}])
              \ \textrm{ or }\    \frac{P(\cdot)}{2\mu+\lambda(\cdot)}
              \ \textrm{ is a monotone function on}\
              [\underline{\rho},\bar{\rho}],
      \end{array}
      \right.
    \end{equation}
   where $0\leq
2\underline{\rho}<\tilde{\rho}<\bar{\rho}$. Then, we can choose
that $P=A\rho^\gamma$ and   $\lambda(\rho)=c\rho^\beta$ with
$\gamma\geq 1$ and
  $\beta\geq0$.
\end{rem}

\begin{rem}
  Using a similar argument as that in \cite{Vaigant}, one can
  obtain the uniqueness of the solution if the initial data
  satisfy $\rho_0\geq\underline{\rho}>0$, and $(\rho_0,u_0)\in C^{1+\alpha}\times C^{2+\alpha}$ or
  $(\rho_0,u_0)\in W^{1,q}\times H^{2}$, where $\alpha\in(0,1)$ and
  $q>2$.
\end{rem}

The proof of Theorem \ref{2VDD-T1.1} consists in the derivation of a
priori estimates for smooth solutions corresponding to mollified
initial data, and the application of these estimates in extracting
limiting weak solutions as the mollifying parameter goes to zero.
Specifically, in Section \ref{2VDD-S2}, we fix a smooth, local in
time solution for which $0\leq\rho\leq\overline{\rho}$ and
$A_1+A_2\leq 2(C_0+C_f)^\theta$, where $\theta\in(0,1)$,
    $$
    A_1(T)=\sup_{0<t\leq T}\sigma\int|\nabla
    u|^2dx+\int^T_0\int\sigma\rho|\dot{u}|^2dxdt
    $$
and
    $$
    A_2(T)=\sup_{0<t\leq
    T}\sigma^2\int\rho|\dot{u}|^2dx+\int^T_0\int\sigma^2|\nabla\dot{u}|^2dxdt,
    $$
then obtain the estimate $A_1+A_2\leq (C_0+C_f)^\theta$, and prove
that the density remains in a compact subset of
$[0,\overline{\rho})$. Using the classical continuation method, we
can close these estimates, which are then applied in Section
\ref{2VDD-S3} to show the solution can be obtained in the limit as
the mollifying parameter goes to zero.

Using the initial condition $u_0\in H^1$, we can obtain pointwise
bounds for $F$ in Proposition \ref{2VDD-P2.8}, which is the key
point of the a priori estimates. Because that the mass equation
can be transformed to the following form,
    \begin{equation}
      \frac{d}{dt}\Lambda(\rho(x(t),t))+P(\rho(x(t),t))-P(\tilde{\rho})=-F(x(t),t),\label{2VDD-E1.18}
    \end{equation}
where $\Lambda$ satisfies that $\Lambda(\tilde{\rho})=0$ and
$\Lambda'(\rho)=\frac{2\mu+\lambda(\rho)}{\rho}$, a curve $x(t)$
satisfies $\dot{x}(t)=u(x(t),t)$, thus pointwise bounds for the
density will therefore follow from  pointwise bounds for $F$. On the
other hand, using a similar argument as that in \cite{Vaigant} and
the estimate $\int^T_0\|F(\cdot,t)\|_{L^\infty}dt\leq C(T)$, we can
obtain the strong limit of approximate densities $\{\rho^\delta\}$,
see Section \ref{2VDD-S3}.

Then, we study the propagation of singularities in solutions
obtained in Theorem \ref{2VDD-T1.1}.  We show that each point of
$\mathbb{R}^2$ determines a unique integral curve of the velocity
field at the initial time $t=0$, and that this system of integral
curves defines a locally bi-H\"{o}lder homeomorphism of any open
subset $\Omega$ onto its image $\Omega^t$ at each time $t>0$. Using
this Lagrangean structure,  we show that if there is a vacuum domain
at the
  initial time, then the vacuum domain will exist for all time, and
  vanishes as time goes to infinity, see Theorem \ref{2VDD-T1.4}.
 Also, we show that, if
the initial density has a limit at a point from a given side of a
continuous hypersurface, then at each later time both the density
and the divergence of the velocity have limits at the transported
point from the corresponding side of the transported hypersurface,
which is also a continuous manifold. If the limits from both sides
exist, then the Rankine-Hugoniot conditions hold in a strict
pointwise sense, showing that the jump in the
$(\lambda+2\mu)\mathrm{div}u$ is proportional to the jump in the
pressure. This leads to a derivation of an explicit representation
for the strength of the jump in $\Lambda(\rho)$ in non-vacuum
domain. These results generalize and improve upon the earlier
results of
    Hoff-Santos \cite{Hoff2008} in a significant way:
the domain may contain the vacuum states.

\begin{thm}\label{2VDD-T1.2}
Assume that the conditions of Theorem \ref{2VDD-T1.1} hold.
\begin{description}
    \item[(1)] For each $t_0\geq0$ and each $x_0\in\mathbb{R}^2$, there is a unique curve
    $X(\cdot;x_0,t_0)\in
    C^1((0,\infty);\mathbb{R}^2)\cap C^{1-\frac{\alpha}{2}}([0,\infty);\mathbb{R}^2)
    $, $\alpha\in(0,1)$, satisfying
        \begin{equation}
          X(t;x_0,t_0)=x_0+\int^t_{t_0} u(X(s;x_0,t_0),s)ds.\label{2VDD-E1.26}
        \end{equation}
    \item[(2)] Denote $X(t,x_0)=X(t;x_0,0)$. For each $t>0$ and any open set $\Omega\subset\mathbb{R}^2$,
    $\Omega^t=X(t,\cdot)\Omega$ is open and the map $x_0\longmapsto
    X(t,x_0)$ is a homeomorphism of $\Omega$ onto $\Omega^t$.
    \item[(3)] For any $0\leq t_1,t_2\leq T$, the map $X(t_1,y)\rightarrow
    X(t_2,y)$ is H\"{o}lder continuous from $\mathbb{R}^2$ onto
    $\mathbb{R}^2$. Specifically, for any $y_1,y_2\in
    \mathbb{R}^2$,
        \begin{equation}
    |X(t_2,y_2)-X(t_2,y_1)|\leq \exp(1-e^{-C (1+T)})|X(t_1,y_2)-X(t_1,y_1)|^{e^{-C
    (1+T)}}.
        \end{equation}
    \item[(4)] Let $\mathcal{M}\subset\mathbb{R}^2$ be a $C^\alpha$
    $1$-manifold, where $\alpha\in[0,1)$. Then, for any $t>0$,
    $\mathcal{M}^t=X(t,\cdot)\mathcal{M}$ is a $C^\beta$
    $1$-manifold, where $\beta=\alpha e^{-C
    (1+t)}$.
\end{description}
\end{thm}

\begin{thm}\label{2VDD-T1.3}
Assume that the conditions of Theorem \ref{2VDD-T1.1} hold. Let $V$
be a nonempty open set in $\mathbb{R}^2$. If
$\mathrm{essinf}\rho_0|_V\geq \underline{\rho}>0$, then there is a
positive number $\underline{\rho}^{-}$ such that,
    $$
    \rho(\cdot,t)|_{V^t}\geq \underline{\rho}^-,
    $$
for all $t>0$, where $V^t=X(t,\cdot)V$.
\end{thm}

\begin{thm}\label{2VDD-T1.4}
Assume that the conditions of Theorem \ref{2VDD-T1.1} hold. Let $U$
be a nonempty open set in $\mathbb{R}^2$. Assume that $\rho_0|_U
=0$. Then,
    $$
    \rho(\cdot,t)|_{U^t}=0,
    $$
for all $t>0$, where $U^t=X(t,\cdot)U$. Furthermore, we have
    \begin{equation}
    \lim_{t\rightarrow\infty}\int|\rho-\tilde{\rho}|^4(x,t)dx=0,\label{2VDD-E1.28}
    \end{equation}
and
     \begin{equation}
    \lim_{t\rightarrow\infty}|\{x\in\mathbb{R}^2|\rho(x,t)=0\}|=0.\label{2VDD-E1.29}
    \end{equation}
\end{thm}
Theorems \ref{2VDD-T1.2}-\ref{2VDD-T1.4} are proved in Section
\ref{2VDD-S4}.

In the following theorem, applying the Lagrangean structure of
Theorem \ref{2VDD-T1.2}, we establish a result concerning the
transport by the velocity field of pointwise continuity of the
density. Recall first that  the oscillation of $g$  at $x$ with
respect to $E$ is defined by (as in \cite{Hoff2008})
    $$
    \mathrm{osc}(g;x,E)=\lim_{R\rightarrow0}
    \left(\mathrm{esssup}g|_{E\cap B_R(x)}-\mathrm{essinf}g|_{E\cap
    B_R(x)}
    \right),
    $$
where $x\in \bar{E}$ and $g$ maps an open set $E\subset
\mathbb{R}^2$ into $\mathbb{R}$. We shall say that $g$ is
continuous at an interior point $x$ of $E$, if osc$(g;x,E)=0$.
\begin{thm}\label{2VDD-T1.5}
  Assume that the conditions of Theorem \ref{2VDD-T1.1} hold. Let
  $E\subset\mathbb{R}^2$ be open  and $x_0\in\bar{E}$. If
  $\mathrm{osc}(\rho_0;x_0,E)=0$, then
  $\mathrm{osc}(\rho(\cdot,t);X(t,x_0),X(t,\cdot)E)=0$. In
  particular, if $x_0\in E$ and $\rho_0$ is continuous at $x_0$,
  then $\rho(\cdot,t)$ is continuous at $X(t,x_0)$.
\end{thm}

Now, let $\mathcal{M}$ be a $C^0$ $1$-manifold in $\mathbb{R}^2$ and
 $x_0\in\mathcal{M}$. Then there is a neighborhood $G$ of $x_0$
which is the disjoint union $G=(G\cap\mathcal{M})\cup E_+\cup E_-$,
where $E_{\pm}$ are open and $x_0$ is a limit point of each. If
osc$(g;x_0,E_+)=0$, then the common value $g(x_0+,t)$ is the
one-sided limit of $g$ at $x_0$ from the plus-side of $\mathcal{M}$,
and similar for the one-sided limit $g(x_0-,t)$ from the minus-side
of $\mathcal{M}$. If both of these limits exist, then the difference
$[g(x_0)]:=g(x_0+)-g(x_0-)$ is the jump in $g$ at $x_0$ with respect
to $\mathcal{M}$. (see \cite{Hoff2008})

Now, we state our main results on the propagation of singularities
in solutions.

\begin{thm}\label{2VDD-T1.6}
  Let $(\rho,u)$ as in Theorem \ref{2VDD-T1.1},
  $\mathcal{M}$ be a $C^0$ 1-manifold and $x_0\in \mathcal{M}$.

  (a) If $\rho_0$ has a one-sided limit at $x_0$ from the
  plus-side of $\mathcal{M}$, then for each $t>0$, $\rho(\cdot,t)$
  and $\mathrm{div}u(\cdot,t)$ have one-sided limits at $X(t,x_0)$
  from the plus-side of the $C^0$ 1-manifold
  $X(t,\cdot)\mathcal{M}$ corresponding to the choice
  $E_+^t=X(t,\cdot)E_+$. The map $t\mapsto \rho(X(t,x_0)+,t)$ is
  in $C^{\frac{3q}{4+4q}}([0,\infty))\cap C^1((0,\infty))$ and the map $t\mapsto \mathrm{div}u(X(t,x_0)+,t)$ is
  locally H\"{o}lder continuous on $(0,\infty)$.

  (b) If both one-sided limits $\rho_0(x_0\pm)$ of $\rho_0$ at
  $x_0$ with respect to $\mathcal{M}$ exist, then for each $t>0$,
  the jumps in $P(\rho(\cdot,t))$ and $\mathrm{div}u(\cdot,t)$ at
  $X(t,x_0)$ satisfy the Rankine-Hugonoit condition
        \begin{equation}
          [(2\mu+\lambda(\rho(X(t,x_0),t)))\mathrm{div}u(X(t,x_0),t)]
          =[P(\rho(X(t,x_0),t))].\label{2VDD-E1.31}
        \end{equation}

  (c) Furthermore, if $\rho_0(x_0\pm)\geq\underline{\rho}>0$, then the jump in
   $\Lambda(\rho)$ satisfies the representation
    \begin{equation}
      [\Lambda(\rho(X(t,x_0),t))]=\exp\left(
      -\int^t_0a(\tau,x_0)d\tau
      \right)[\Lambda(\rho_0(x_0))]
    \end{equation}
  where
  $a(t,x_0)=\frac{[P(\rho(X(t,x_0),t))]}{[\Lambda(\rho(X(t,x_0),t))]}$.
\end{thm}
Theorems \ref{2VDD-T1.5}-\ref{2VDD-T1.6} are proved in Section
\ref{2VDD-S5}.
\begin{rem}
From Theorems \ref{2VDD-T1.2}-\ref{2VDD-T1.6}, we have that if
$[\rho_0]$ is nonzero at $x_0$, then $[\rho(\cdot,t)]$ is nonzero
at $X(t;x_0,0)$ for every $t>0$. That is, singularities of this
type persist for all time. On the other hand,  if $P'(\rho)>0$
with $\underline{\rho}^-\leq \rho\leq \overline{\rho}$, then $a$
is strictly positive and the jump in $\Lambda(\rho)$ in the
non-vacuum domain decays exponentially in time. Thus, the jumps in
$\rho$, in $\frac{P(\rho)}{\lambda(\rho)+2\mu}$ and in
$\mathrm{div} u$  in the non-vacuum domain decay exponentially in
time as well.
\end{rem}

At last, we will show that the condition of $\mu=$constant will
induce
   a singularity of the system at vacuum in the following
   two aspects: 1) considering the special case where two fluid regions initially
separated by a vacuum region, we show that the solution we obtained
is a nonphysical weak solution in which separate
   kinetic energies of the two fluids need not to be conserved; 2) we show the blow-up of
    smooth solutions for the spherically
symmetric system when the initial density is compactly supported.
Thus, the viscosity coefficient $\mu$ plays a key role in the
Navier-Stokes equations.

If we consider the special case where two fluid regions initially
separated by a vacuum region, Theorem \ref{2VDD-T1.4} shows that the
solution obtained in
   Theorem \ref{2VDD-T1.1} is a nonphysical weak solution in which
   the two fluids cannot collide independent of their initial
   velocities. In the following, we will show that the separate
   kinetic energies of the two fluids needn't  to be conserved.

If the initial data are spherically symmetric, i.e,
    \begin{equation}
    \rho_0(x)=\varrho_0(r),\ u_0(x)=v_0(r)\frac{x}{r},\ r=|x|,\label{2VDD-E1.28-1}
    \end{equation}
and the external force $f \equiv0$, from Theorem \ref{2VDD-T1.1},
one can prove that the system (\ref{2VDD-E1.1})--(\ref{2VDD-E1.2})
has a spherically symmetric solution $(\rho,u)$ satisfying
    \begin{equation}
    \rho(x,t)=\varrho(r,t),\ u(x,t)=v(r,t)\frac{x}{r},\
    r=|x|.\label{2VDD-E1.29-1}
    \end{equation}
Then $(\varrho,v)$ is a solution of the following system
    \begin{equation}
      \left\{\begin{array}{l}
      \partial_t \varrho+\partial_r(\varrho v)+\frac{1}{r}\varrho v=0,\\
           \varrho(\partial_t v+v\partial_rv)
      +\partial_r P(\varrho)=\partial_r\{(\lambda(\varrho)+2\mu)(\partial_r v+\frac{v}{r})\},
     \end{array}
      \right.\label{2VDD-E1.30-1}
    \end{equation}
         \begin{equation}
      v(r,t)\rightarrow0,
      \ \varrho(r,t)\rightarrow \tilde{\rho}>0, \ \textrm{ as
      }r\rightarrow\infty, \ t>0,
    \end{equation}
    \begin{equation}
      (\varrho,v)|_{t=0}=(\varrho_0,v_0).\label{2VDD-E1.30}
    \end{equation}
Furthermore, assume that there are two positive constant $0<a<b$,
such that
    $$
    \varrho_0(r)=0,\ r\in(a,b),
    $$
and
    $$
        \varrho_0(r)\geq \underline{\varrho}>0,\ r\in (0,a)\cup(b,+\infty).
    $$
Then, from  Theorems \ref{2VDD-T1.2}-\ref{2VDD-T1.4}, we have
there are two curves $a(t)$ and $b(t)$ satisfying
    $$
    0<a(t)<b(t)<\infty,
    $$
    $$
        a(0)=a,\ a'(t)=v(a(t),t),
        \ b(0)=b,\ b'(t)=v(b(t),t),
    $$
    $$
    \varrho(r,t)=0,\ r\in(a(t),b(t)),
    $$
and
    $$
        \varrho(r,t)\geq \underline{\varrho}^->0,\ r\in (0,a(t))\cup
        (b(t),+\infty),
    $$
for some positive constant $\underline{\varrho}^-$. Using a similar
argument as that in \cite{hoff}, we can obtain the following
theorem.

\begin{thm}\label{2VDD-T1.7}
  Assume that the conditions of Theorem \ref{2VDD-T1.1} hold, $\int^1_0 s^{-2}P(s)ds<\infty$, $f=0$,  and
  the initial data satisfy (\ref{2VDD-E1.28-1}), then we have
        \begin{equation}
         \frac{d}{dt}E(t)=2(\lambda(0)+2\mu)a(t)v(a(t),t)
    \frac{a(t)v(a(t),t)-b(t)v(b(t),t)}{a^2(t)-b^2(t)}   ,
          \ t\geq0,\label{2VDD-E1.33}
        \end{equation}
  where
    $$
    E(t)=\int^{a(t)}_{0}\left(\frac{1}{2}\varrho v^2+\overline{G}(\varrho)
    \right)rdr+\int^t_0\int_0^{a(s)}(\lambda+2\mu)(v_r+\frac{v}{r})^2rdrds,
    $$
  and
    $$
      \overline{G}(\rho)=\rho\int^{\rho}_{0}\frac{P(s)}{s^2}ds.
    $$
\end{thm}

\begin{rem}
If the
 viscosity coefficient $\mu$ is a function of the density and
 $\lambda(0)=\mu(0)=0$, the equality (\ref{2VDD-E1.33}) implies the separate
   kinetic energies of the two fluids are conserved. Thus,
   the main reason for the appearance of non-physical solutions  comes from the
 viscosity coefficient $\mu$ being independent of the density.
\end{rem}

\begin{rem}
  The physical solution of the
  (\ref{2VDD-E1.30-1})--(\ref{2VDD-E1.30}) may be  obtained  by
  constructing separately the solutions for each of the fluids
  $\varrho_0|_{[0,a]}$ and $\varrho_0|_{[b,\infty)}$ with the boundary
  conditions
    \begin{equation}
    (\lambda(\varrho)+2\mu)(v_r+\frac{v}{r})=P(\varrho),
    \ r=a(t),\    b(t).\label{2VDD-E1.34}
    \end{equation}
  When $a(t)<b(t)$, one can obtain the composite solution
  $(\varrho,v)$. Observe that the kinetic energies are separately
  conserved because of the boundary conditions (\ref{2VDD-E1.34}).
  When $v_0$ is large and positive on $[0,a]$, large and
  negative on $[b,\infty)$, a collision $a(t)=b(t)$ may  occur in finite time.
\end{rem}

Finally, in Section \ref{2VDD-S7}, we will give a non-global
existence theorem on smooth solutions for the spherically
symmetric system when the initial density is of compact support.
The corresponding theorem on compressible Navier-Stokes equations
with constant viscosity and heat conductivity coefficients was
obtained in \cite{xin}. Here we generalize the above theorem to
the case when the second viscosity coefficient depends on the
density for the isentropic gas flow.

\begin{thm}\label{2VDD-T1.8}
  Suppose that $(\rho,u)\in C^1([0,T];H^k)$, $k>3$ is a spherically
  symmetric solution to the Cauchy problem (\ref{2VDD-E1.1}) and
  (\ref{2VDD-E1.2}) with $f=0$. Assume that $P(\rho)=A\rho^\gamma$ and $\lambda(\rho)=c\rho^\beta$ with
  $1< \beta\leq \gamma$ and $A,c>0$.
  If the support of the initial density
  $\rho_0$ is compact and $\rho_0\not\equiv0$, then $T$ must be
  finite.
\end{thm}

We now briefly review some previous works about the Navier-Stokes
equations with density-dependent viscosity coefficients. For the
free boundary problem of one-dimensional or spherically symmetric
isentropic fluids,  there are many works, please see
\cite{fang00,jiang,liu,yang3,zhang1}  and the references cited
therein. Under a special condition between $\mu$ and $\lambda$,
$\lambda=2\rho\mu'-2\mu$, there are some existence results of global
weak solutions for the system with the Korteweg stress tensor or the
additional quadratic friction term, see \cite{bresch1,bresch2}. H.
L. Li, J. Li and Z. P. Xin \cite{Li2008} showed a very interesting
result that for any global entropy weak solution of the
one-dimensional system, any vacuum state must vanish within finite
time. Also see Lions \cite{Lions} for multidimensional isentropic
fluids.

  We should mention that the methods
    introduced by   Hoff-Santos in \cite{Hoff2008},  Hoff in
    \cite{Hoff1995} and Vaigant-Kazhikhov in \cite{Vaigant}
   will play a crucial role
in our proof here.

\section{A priori estimates}\label{2VDD-S2}
In this section, we derive some a priori estimates for local smooth
solutions of the system (\ref{2VDD-E1.1})--(\ref{2VDD-E1.2}) with
strictly positive densities. Thus, we fix a smooth solution
$(\rho,u)$ of (\ref{2VDD-E1.1})--(\ref{2VDD-E1.2}) on
$\mathbb{R}^2\times[0,T]$ for some time $T>0$, with smooth initial
data $(\rho_0,u_0)$ and smooth external force $f$, satisfying
        \begin{equation}
      0\leq  \rho\leq\bar{\rho}\label{2VDD-E2.1-1}
        \end{equation}
and
    \begin{equation}
      A_1+A_2\leq 2(C_0+C_f)^\theta.\label{2VDD-E2.1-2}
    \end{equation}
In this paper, we choose $\theta=\frac{1}{2}$ and assume that
$\varepsilon\leq 1$.

Before proceeding, we remark that a careful application of the
standard Rankine-Hugoniot condition to (\ref{2VDD-E1.1}) shows
that discontinuities in $\rho$, $P(\rho)$ and $\nabla u$ across
hypersurfaces can be expected to persist for all time, but that
the functions $F$ and $w$ should be relatively smooth in positive
time reflecting a cancellation of singularities (for example, see
\cite{Hoff2008,Hoff2005,Hoff1995,Hoff1995-2}). We can rewrite the
momentum equation in the form,
  \begin{equation}
    \rho\dot{u}^j=\partial_j F+\mu \partial_k w^{j,k}+\rho
    f^j.\label{2VDD-E1.19}
    \end{equation}
Thus $L^2$ estimates for $\rho\dot{u}$, immediately imply $L^2$
hounds for $\nabla F$ and $\nabla w$. Stated differently, the
decomposition (\ref{2VDD-E1.19}) implies that
        \begin{equation}
          \Delta F=\mathrm{div}(\rho\dot{u}-\rho f).\label{2VDD-E1.20}
        \end{equation}
These two relations (\ref{2VDD-E1.19})--(\ref{2VDD-E1.20}) will play
the important role in this section.
\begin{prop}\label{2VDD-P2.1}
  There is a positive constant $C=C(\bar{\rho})$ such that, if
  $(\rho,u)$ is a smooth solution of (\ref{2VDD-E1.1})--(\ref{2VDD-E1.2})
  satisfying (\ref{2VDD-E2.1-1})--(\ref{2VDD-E2.1-2}), then
        \begin{equation}
          \sup_{0\leq t\leq T}\int\left[
          \frac{1}{2}\rho|u|^2+G(\rho)
          \right]dx+\int^T_0\int|\nabla u|^2dxdt\leq
          C(C_0+C_f).\label{2VDD-E2.1}
        \end{equation}
\end{prop}
\begin{proof}
  Using the energy estimate, we can easily obtain
  (\ref{2VDD-E2.1}), and omit the details.
\end{proof}

The following lemma contains preliminary versions of $L^2$ bounds
for $\nabla u$ and $\rho\dot{u}$.

\begin{lem}\label{2VDD-L2.1}
  If $(\rho,u)$
  is a smooth solution of (\ref{2VDD-E1.1})--(\ref{2VDD-E1.2}) as
  in Proposition \ref{2VDD-P2.1}, then there is a constant $C=C(\bar{\rho})$ such that
    \begin{equation}
      \sup_{0<t\leq T}\sigma\int|\nabla
      u|^2dx+\int^T_0\int\sigma\rho|\dot{u}|^2dxdt
            \leq C\left(C_0+C_f+O_1
            \right),\label{2VDD-E2.4}
    \end{equation}
where $O_1=\int^T_0\int\sigma|\nabla u|^3dxdt$, and
    \begin{eqnarray}
      &&\sup_{0<t\leq T}\sigma^2\int\rho|\dot{u}|^2dx
      +\int^T_0\int\sigma^2|\nabla \dot{u}|^2dxdt\nonumber\\
            &\leq& C\left(C_0+C_f+A_1(T)
            \right)+C\int^T_0\int\sigma^2(|u|^4+|\nabla
            u|^4)dxdt.\label{2VDD-E2.5}
    \end{eqnarray}
\end{lem}
\begin{proof}
Multiplying (\ref{2VDD-E1.1})$_2$ by $\sigma\dot{u}$, integrating it
over $\mathbb{R}^2\times[0,t]$, we obtain
    \begin{eqnarray}
      &&\int^t_0\int\sigma\rho|\dot{u}|^2dxds\nonumber\\
            &=&\int^t_0\int\left(-\sigma\dot{u}\cdot\nabla P+\mu\sigma\Delta
            u\cdot\dot{u}+\sigma\nabla((\lambda+\mu)\mathrm{div}u)\cdot\dot{u}
            +\sigma\rho f\cdot\dot{u}
            \right)dxds\nonumber\\
                &:=&\sum^4_{i=1}J_i.\label{2VDD-E2.6}
    \end{eqnarray}
Using the integration by parts and H\"{o}lder's inequality, we have
    \begin{eqnarray}
      J_1&=&-\int^t_0\int\sigma\dot{u}\cdot\nabla P dxds
      \nonumber\\
            &=&-\int^t_0\int\left(-\sigma(\mathrm{div}u)_t(P-P(\tilde{\rho}))
            +\sigma(u\cdot\nabla u)\cdot\nabla P\right)dxds
            \nonumber\\
      &=&\int\sigma\mathrm{div}u(P-P(\tilde{\rho}))dx\nonumber\\
            &&-
      \int^t_0\int\left(\sigma_t\mathrm{div}u(P-P(\tilde{\rho}))
      +\sigma P_t\mathrm{div}u
            +\sigma(u\cdot\nabla u)\cdot\nabla P\right)dxds
            \nonumber\\
      &=&\int\sigma\mathrm{div}u(P-P(\tilde{\rho}))dx-
      \int^t_0\int\{\sigma_t\mathrm{div}u(P-P(\tilde{\rho}))\nonumber\\
            &&\left.
      +\sigma P'\left(-\rho(\mathrm{div}u)^2-u\cdot\nabla
      \rho\mathrm{div}u
            +(u\cdot\nabla u)\cdot\nabla \rho\right)\right\}dxds
            \nonumber\\
      &=&\int\sigma\mathrm{div}u(P-P(\tilde{\rho}))dx-
      \int^t_0\int\{\sigma_t\mathrm{div}u(P-P(\tilde{\rho}))\nonumber\\
            &&      -\sigma P'\rho(\mathrm{div}u)^2+\sigma P
      (\mathrm{div}u)^2-\sigma P\partial_i u^j\partial_j u^i\}dxds
            \nonumber\\
      &\leq& C\left(\sigma\int|\nabla u||\rho-\tilde{\rho}|dx
      +\int^{1\wedge t}_0\int |\nabla u||\rho-\tilde{\rho}|dxds
      +\int^t_0\int |\nabla u|^2dxds\right),
    \end{eqnarray}
        \begin{eqnarray}
          J_2&=&\int^t_0\int\sigma\mu \Delta
          u\cdot\dot{u}dxds\nonumber\\
                &=&-\frac{\sigma\mu}{2}\int|\nabla u|^2dx
                +\frac{\mu}{2}\int^{1\wedge t}_0\int |\nabla
                u|^2dxds
                -\int^t_0\int\sigma\mu\partial_i u^j\partial_i(u^k\partial_k
                u^j)dxds\nonumber\\
          &\leq&-\frac{\sigma\mu}{2}\int|\nabla u|^2dx
                +\frac{\mu}{2}\int^{1\wedge t}_0\int |\nabla
                u|^2dxds
                +CO_1,
        \end{eqnarray}
            \begin{eqnarray}
              J_3&=&\int^t_0\int\sigma\nabla((\lambda+\mu)\mathrm{div}u)\cdot\dot{u}dxds\nonumber\\
              &=&-\frac{\sigma}{2}\int(\lambda+\mu)|\mathrm{div}u|^2dx
              +\int^t_0\int(\frac{\sigma}{2}\lambda'\rho_t|\mathrm{div}u|^2
              +\frac{\sigma_t}{2}(\lambda+\mu)|\mathrm{div}u|^2)dxds\nonumber\\
                    &&+\int^t_0\int\sigma\nabla((\lambda+\mu)\mathrm{div}u)\cdot(u\cdot\nabla u)dxds
                    \nonumber\\
                        &\leq&-\frac{\sigma}{2}\int(\lambda+\mu)|\mathrm{div}u|^2dx
              +\int^{1\wedge
              t}_0\int\frac{1}{2}(\lambda+\mu)|\mathrm{div}u|^2dxds+CO_1,
            \end{eqnarray}
        \begin{equation}
          J_4=\int^t_0\int\sigma\rho f\cdot\dot{u} dxds\leq
          \frac{1}{2}\int^t_0\int\sigma\rho |\dot{u}|^2 dxds
          +C\int^t_0\int\sigma f^2dxds,\label{2VDD-E2.10}
        \end{equation}
where $t\in[0,T]$. From (\ref{2VDD-E2.1}),
(\ref{2VDD-E2.6})--(\ref{2VDD-E2.10}), we immediately obtain
(\ref{2VDD-E2.4}).

Takeing the operator $\partial_t+\mathrm{div}( u\cdot)$ in
(\ref{2VDD-E1.1})$_2$, multiplying
 by $\sigma^2\dot{u}$ and integrating, we obtain
    \begin{eqnarray}
      &&\frac{\sigma^2}{2}\int\rho|\dot{u}|^2dx\nonumber\\
            &=&\int^t_0\int\{\sigma
            \sigma'\rho|\dot{u}|^2-\sigma^2\dot{u}^j[
            \partial_j P_t+\mathrm{div}(\partial_j Pu)]
            +\mu\sigma^2\dot{u}^j[\Delta u^j_t+\mathrm{div}(u\Delta
            u^j)]\nonumber\\
                &&+\sigma^2\dot{u}^j[\partial_j\partial_t((\lambda+\mu)\mathrm{div}u)
                +\mathrm{div}(u\partial_j((\lambda+\mu)\mathrm{div}u))]
                +\sigma^2\dot{u}^j[(\rho f^j)_t+\mathrm{div}(u\rho f^j)]
            \}dxds\nonumber\\
                &:=&\sum^5_{i=1}K_i.\label{2VDD-E2.11}
    \end{eqnarray}
Using the integration by parts and H\"{o}lder's inequality, we have
    \begin{eqnarray}
      K_2&=&-\int^t_0\int\sigma^2\dot{u}^j[
            \partial_j P_t+\mathrm{div}(\partial_j
            Pu)]dxds\nonumber\\
                &=&\int^t_0\int\sigma^2[\partial_j\dot{u}^jP'\rho_t
                +\partial_k\dot{u}^j\partial_jP
                u^k]dxds\nonumber\\
      &=&\int^t_0\int\sigma^2[
        -P'\rho\mathrm{div}u\partial_j\dot{u}^j+\partial_k(\partial_j\dot{u}^ju^k)P
      -P\partial_j(\partial_k\dot{u}^ju^k)]dxds\nonumber\\
            &\leq&C\left(\int^t_0\int|\nabla u|^2
            dxds    \right)^\frac{1}{2}
            \left(\int^t_0\int\sigma^2|\nabla\dot{u}|^2
            dxds    \right)^\frac{1}{2}\nonumber\\
            &\leq&C (C_0+C_f)^\frac{1}{2}
            \left(\int^t_0\int\sigma^2|\nabla\dot{u}|^2
            dxds    \right)^\frac{1}{2},
    \end{eqnarray}
        \begin{eqnarray}
          K_3&=&\int^t_0\int\mu\sigma^2\dot{u}^j[\Delta u^j_t+\mathrm{div}(u\Delta
            u^j)]dxds\nonumber\\
                &=&-\int^t_0\int\sigma^2\mu[
                \partial_i\dot{u}^j\partial_iu^j_t+\Delta u^j
                u\cdot\nabla \dot{u}^j
                ]dxds\nonumber\\
            &=&-\int^t_0\int\sigma^2\mu[
            |\nabla \dot{u}|^2-\partial_i\dot{u}^j u^k\partial_k\partial_i
            u^j-\partial_i \dot{u}^j\partial_i u^k\partial_k u^j
            +\Delta u^ju\cdot\nabla\dot{u}^j
            ]dxds\nonumber\\
                    &=&-\int^t_0\int\sigma^2\mu[|\nabla \dot{u}|^2
                    +\partial_i\dot{u}^j\partial_k
                    u^k\partial_iu^j-\partial_i \dot{u}^j\partial_i u^k\partial_k u^j-\partial_i u^j\partial_i
                    u^k\partial_k\dot{u}^j
                    ]dxds\nonumber\\
            &\leq&-\frac{1}{2}\int^t_0\int\sigma^2\mu|\nabla \dot{u}|^2dxds
            +C\int^t_0\int \sigma^2|\nabla u|^4dxds,
        \end{eqnarray}
            \begin{eqnarray}
                K_4&=&\int^t_0\int\sigma^2\dot{u}^j[\partial_j\partial_t((\lambda+\mu)\mathrm{div}u)
                +\mathrm{div}(u\partial_j((\lambda+\mu)\mathrm{div}u))]dxds\nonumber\\
            &=&-\int^t_0\int\{\sigma^2\partial_j\dot{u}^j[\partial_t((\lambda+\mu)\mathrm{div}u)
                +\mathrm{div}(u(\lambda+\mu)\mathrm{div}u)]\nonumber\\
                    &&+\sigma^2\dot{u}^j\mathrm{div}(\partial_ju(\lambda+\mu)\mathrm{div}u)\}dxds\nonumber\\
                        &=&-\int^t_0\int\sigma^2\partial_j\dot{u}^j[\partial_t((\lambda+\mu)\mathrm{div}u)
                +u^k\lambda'\partial_k\rho\mathrm{div}u+(\lambda+\mu)u^k\partial_k\mathrm{div}u]dxds+O_4\nonumber\\
                                &=&-\int^t_0\int\sigma^2\partial_j\dot{u}^j[(\lambda+\mu)\frac{D}{Dt}\mathrm{div}u
                                +\lambda'\rho_t\mathrm{div}u
                +u^k\lambda'\partial_k\rho\mathrm{div}u]dxds+O_4\nonumber\\
                             &=&-\int^t_0\int\sigma^2\partial_j(\partial_tu^j
                             +u\cdot\nabla u^j)(\lambda+\mu)\frac{D}{Dt}
                             \mathrm{div}udxds+O_4\nonumber\\
            &=&-\int^t_0\int\sigma^2(\lambda+\mu)|\frac{D}{Dt}
                             \mathrm{div}u|^2dxds+O_4,
            \end{eqnarray}
        \begin{eqnarray}
          K_5&=&\int^t_0\int \sigma^2\dot{u}^j[(\rho f^j)_t+\mathrm{div}(u\rho f^j)]
           dxds\nonumber\\
                &=&\int^t_0\int \sigma^2\dot{u}^j[\rho f^j_t+\rho u\cdot\nabla
                f^j]dxds\nonumber\\
                    &\leq& C\int^t_0\int[\sigma\rho|\dot{u}|^2
                    +\sigma^2|u|^4+\sigma^4|\nabla
                    f|^4+\sigma^3|f_t|^2
                    ]          dxds,\label{2VDD-E2.15}
        \end{eqnarray}
where $O_4$ denotes any term dominated by
$C\int^t_0\int\sigma^2|\nabla u|^2(|\nabla
\dot{u}|+|\frac{D}{Dt}\mathrm{div}u|)$, and $t\in[0,T]$. From
(\ref{2VDD-E2.11})--(\ref{2VDD-E2.15}), we immediately obtain
(\ref{2VDD-E2.5}).
\end{proof}

The following  lemmas will be applied to bound the higher order
terms occurring on the right hand sides of
(\ref{2VDD-E2.4})--(\ref{2VDD-E2.5}).

\begin{lem}\label{2VDD-L2.2}
  If $(\rho,u)$
  is a smooth solution of (\ref{2VDD-E1.1})--(\ref{2VDD-E1.2}) as
  in Proposition \ref{2VDD-P2.1}, then there is a constant $C=C(\bar{\rho})$ such that,
    \begin{equation}
      \|u\|_{L^p}^p\leq
    C_p(C_0+C_f)\|\nabla
      u\|_{L^2}^{p-2}+C_p(C_0+C_f)\|\nabla
      u\|_{L^2}^p,\ p\in[2,\infty),\label{2VDD-E2.16-0}
    \end{equation}
        \begin{equation}
          \|\nabla u\|_{L^p}\leq C_p(\|F\|_{L^p}
          +\|w\|_{L^p}+\|P-P(\tilde{\rho})\|_{L^p}
          ),\ p\in(1,\infty),\label{2VDD-E2.16-1}
        \end{equation}
      \begin{equation}
        \|\nabla F\|_{L^p}+    \|\nabla w\|_{L^p}\leq C_p(\|\rho\dot{u}\|_{L^p}
    +\|f\|_{L^p}),\ p\in(1,\infty),\label{2VDD-E2.16-2}
    \end{equation}
         \begin{eqnarray}
        \|F\|_{L^p}+    \| w\|_{L^p}&\leq& C_p\big(\|\rho\dot{u}\|_{L^2}^{\frac{p-2}{p}}
        (\|\nabla
        u\|_{L^2}^{\frac{2}{p}}+\|P-P(\tilde{\rho})\|_{L^2}^{\frac{2}{p}})\nonumber\\
                &&
    +\|\nabla u\|_{L^2}+\|f\|_{L^2}+\|P-P(\tilde{\rho})\|_{L^2}\big),\ p\in[2,\infty).\label{2VDD-E2.16-3}
    \end{eqnarray}
Also, for $0\leq t_1\leq t_2\leq T$, $p\geq2$ and $s\geq0$,
    \begin{equation}
      \int^{t_2}_{t_1}\int\sigma^s|\rho-\tilde{\rho}|^pdxds
      \leq C\left(\int^{t_2}_{t_1}\int\sigma^s|F|^pdxds
      +C_0+C_f
      \right).\label{2VDD-E2.16-4}
    \end{equation}
\end{lem}
\begin{proof}
Using the Galiardo-Nirenberg inequality, we have
    \begin{equation}
      \|u\|_{L^p}\leq C_p\|u\|_{L^2}^\frac{2}{p}\|\nabla
      u\|_{L^2}^{\frac{p-2}{p}},\ p\in[2,\infty).\label{2VDD-E2.16}
    \end{equation}
Since
    $$
    \tilde{\rho}\int|u|^2dx\leq \int\rho |u|^2dx+
    \left(\int|\rho-\tilde{\rho}|^2
    dx    \right)^\frac{1}{2}
    \left(\int|u|^4
    dx    \right)^\frac{1}{2},
    $$
applying (\ref{2VDD-E2.16}) and Proposition \ref{2VDD-P2.1}, we
get
    $$
    \|u\|_{L^2}^2\leq
    C(C_0+C_f)+C(C_0+C_f)^{\frac{1}{2}}\|u\|_{L^2}\|\nabla u\|_{L^2}
    $$
and
     $$
    \|u\|_{L^2}^2\leq
    C(C_0+C_f)+C(C_0+C_f)\|\nabla u\|_{L^2}^2.
    $$
Combining it with (\ref{2VDD-E2.16}), we get (\ref{2VDD-E2.16-0}).

Since $F=(\lambda+2\mu)\mathrm{div}u-P+P(\tilde{\rho})$, we have
    \begin{equation}
    \Delta
    u^j=\partial_j(\frac{F+P-P(\tilde{\rho})}{\lambda+2\mu})+\partial_i(w^{j,i}).
    \label{2VDD-E2.24}
    \end{equation}
From the standard elliptic theory, we can get (\ref{2VDD-E2.16-1}).

We compute from (\ref{2VDD-E1.1})$_2$ that
    $$
    \mu\Delta w^{j,k}=\partial_k(\rho \dot{u}^j)
    -\partial_j(\rho\dot{u}^k)+\partial_j(\rho f^k)-\partial_k(\rho
    f^j).
    $$
Using the standard elliptic theory, we can get
    $$
    \|\nabla w\|_{L^p}\leq C_p(\|\rho\dot{u}\|_{L^p}
    +\|f\|_{L^p}),\ p\in(1,\infty).
    $$
From (\ref{2VDD-E1.19}), we have
    $$
    \partial_j F=\rho\dot{u}^j-\mu\partial_k w^{j,k}-\rho f^j
    $$
and
    $$
    \|\nabla F\|_{L^p}\leq C(\|\rho\dot{u}\|_{L^p}
    +\|\nabla w\|_{L^p}+\|f\|_{L^p}),\ p\in(1,\infty).
    $$
Thus, we can obtain (\ref{2VDD-E2.16-2}).

From (\ref{2VDD-E2.16-2}) with $p=2$ and (\ref{2VDD-E2.16}), we
can immediately obtain (\ref{2VDD-E2.16-3}).

Multiplying the mass equation (\ref{2VDD-E1.1})$_1$ by
$p\sigma^s|\rho-\tilde{\rho}|^{p-1}\mathrm{sgn}(\rho-\tilde{\rho})$,
we have
    \begin{eqnarray*}
    &&(\partial_t+\mathrm{div}(u\cdot))(\sigma^s|\rho-\tilde{\rho}|^p)
    +\frac{\sigma^s}{\lambda+2\mu}((p-1)\rho+\tilde{\rho})|\rho-\tilde{\rho}|^{p-1}|P-P(\tilde{\rho})|\\
        &=&s\sigma^{s-1}\sigma_t|\rho-\tilde{\rho}|^p-
        \frac{\sigma^s}{\lambda+2\mu}((p-1)\rho+\tilde{\rho})\mathrm{sgn}(\rho-\tilde{\rho})
        |\rho-\tilde{\rho}|^{p-1}F.
    \end{eqnarray*}
Integrating, we get
    \begin{eqnarray}
      &&\left.\int\sigma^s|\rho-\tilde{\rho}|^pdx\right|^{t_2}_{t_1}
      +C^{-1}\int^{t_2}_{t_1}\int\sigma^s|\rho-\tilde{\rho}|^pdxds\nonumber\\
        &\leq& C\left(
        \int^{\sigma(t_2)\vee
        t_1}_{t_1}\int\sigma^{s-1}|\rho-\tilde{\rho}|^pdxds
        +\int^{t_2}_{t_1}\int\sigma^s|F|^pdxds
        \right).\label{2VDD-E2.25}
    \end{eqnarray}
Using the result of Proposition \ref{2VDD-P2.1}, we can
immediately get (\ref{2VDD-E2.16-4}).
\end{proof}

Now, we apply the estimates of Lemma \ref{2VDD-L2.2} to close the
bounds in  Lemma \ref{2VDD-L2.1}.

\begin{prop}\label{2VDD-P2.2}
If $(\rho,u)$
  is a smooth solution of (\ref{2VDD-E1.1})--(\ref{2VDD-E1.2}) as
  in Proposition \ref{2VDD-P2.1}, and $\varepsilon$ is small enough,
  then we have
    \begin{equation}
      \sup_{0<t\leq T}\int\left(\sigma
      |\nabla
      u|^2+\sigma^2\rho|\dot{u}|^2\right)dx
            +\int^T_0\int\left(\sigma\rho|\dot{u}|^2+
            \sigma^2|\nabla\dot{u}|^2
            \right)dxdt
        \leq(C_0+C_f)^\theta.\label{2VDD-E2.22}
    \end{equation}
\end{prop}
\begin{proof}
  From Proposition \ref{2VDD-P2.1} and Lemmas
  \ref{2VDD-L2.1}-\ref{2VDD-L2.2}, we have
        \begin{equation}
          \textrm{LHS of }(\ref{2VDD-E2.22})\leq C(C_0+C_f)
          +C\int^T_0\int\left(\sigma|\nabla u|^3
          +\sigma^2|u|^4+\sigma^2|\nabla u|^4
          \right)dxds.\label{2VDD-E2.25-0}
        \end{equation}
From  (\ref{2VDD-E2.16-1}), we get
    \begin{equation}
      \int^T_0\int\sigma^2|\nabla u|^4dxds\leq
      \int^T_0\int\sigma^2\left[
      |F|^4+|w|^4+|P-P(\tilde{\rho})|^4
      \right]dxds.\label{2VDD-E2.26}
    \end{equation}
From (\ref{2VDD-E2.1}), (\ref{2VDD-E2.16-0}) and
(\ref{2VDD-E2.16-2})--(\ref{2VDD-E2.16}), we obtain
    \begin{eqnarray}
        && \int^T_0\int\sigma^2|F|^4dxds\nonumber\\
                &\leq&
      C\int^T_0\sigma^2\left(\int |F|^2dx
      \right)\left(\int |\nabla F|^2dx
      \right)ds\nonumber\\
            &\leq&C\sup_{0\leq t\leq T}\int\sigma(|\nabla u|^2+|\rho-\tilde{\rho}|^2)dx
            \int^T_0\int\sigma\left(\rho|\dot{u}|^2+|f|^2
            \right)dxds\nonumber\\
                &\leq&C(A_1+C_0+C_f)^2,\label{2VDD-E2.28}
    \end{eqnarray}
        \begin{eqnarray}
        && \int^T_0\int\sigma^2|w|^4dxds\nonumber\\
                &\leq&
      C\int^T_0\sigma^2\left(\int |w|^2dx
      \right)\left(\int |\nabla w|^2dx
      \right)ds\nonumber\\
            &\leq&C\sup_{0\leq t\leq T}\int\sigma|\nabla u|^2dx
            \int^T_0\int\sigma\left(\rho|\dot{u}|^2+|f|^2
            \right)dxds\nonumber\\
                &\leq&C(A_1+C_0+C_f)^2,
    \end{eqnarray}
        \begin{eqnarray}
        \int^T_0\int\sigma^2|\rho-\tilde{\rho}|^4dxds&\leq&
        C\left(\int^T_0\int\sigma^2|F|^4dxds+C_0+C_f
        \right)\nonumber\\
            &\leq&C(A_1+C_0+C_f)^2+C(C_0+C_f),
        \end{eqnarray}
    \begin{eqnarray}
        \int^T_0\int\sigma^2|u|^4dxds&\leq&
        C(C_0+C_f)\int^T_0\sigma^2\left(\|\nabla u\|_{L^2}^2+\|\nabla
        u\|_{L^2}^4
        \right)ds\nonumber\\
            &\leq&CA_1(C_0+C_f)^2+C(C_0+C_f)^2.\label{2VDD-E2.30}
        \end{eqnarray}
From (\ref{2VDD-E2.26})--(\ref{2VDD-E2.30}), we have
    \begin{equation}
      \int^T_0\int\sigma^2(|u|^4+|\nabla u|^4)dxds
      \leq C A_1^2+C(C_0+C_f).\label{2VDD-E2.31}
    \end{equation}
    Similarly, we get
     \begin{eqnarray}
      \int^T_0\int\sigma|\nabla u|^3dxds&\leq&
      \int^T_0\int\left(\sigma^2|\nabla u|^4+|\nabla u|^2
      \right)dxds\nonumber\\
        &\leq& CA_1^2+C(C_0+C_f).\label{2VDD-E2.32}
    \end{eqnarray}
Then, from (\ref{2VDD-E2.1-2}), (\ref{2VDD-E2.25-0}) and
(\ref{2VDD-E2.31})--(\ref{2VDD-E2.32}), we obtain
     \begin{eqnarray}
          \textrm{LHS of }(\ref{2VDD-E2.22})&\leq& C(C_0+C_f)
          +CA_1^2\nonumber\\
                &\leq& C(C_0+C_f)
                +C(C_0+C_f)^{2\theta}\nonumber\\
          &\leq&
          (C_0+C_f)^\theta ,
        \end{eqnarray}
when
    \begin{equation}
      \varepsilon^{1-\theta}+\varepsilon^{\theta}
      \leq\frac{1}{2C}.
    \end{equation}
\end{proof}

Then, we consider the H\"{o}lder continuity of $u$ in the
following lemma.
\begin{lem}
  When $t\in(0,T]$ and $\alpha\in(0,1)$, we have
    \begin{equation}
    <u(\cdot,t)>^\alpha\leq C\left(\|\rho\dot{u}\|_{L^2}^{\alpha}
        (\|\nabla
        u\|_{L^2}^{1-\alpha}+(C_0+C_f)^{\frac{1-\alpha}{2}})
    +\|\nabla u\|_{L^2}+(C_0+C_f)^\frac{1-\alpha}{2}
         \right).\label{2VDD-E2.35}
    \end{equation}
\end{lem}
\begin{proof} Let  $p=\frac{2}{1-\alpha}$.
 From (\ref{2VDD-E2.16-1}),
  (\ref{2VDD-E2.16-3}) and Sobolev's
embedding theorem, we have
    \begin{eqnarray*}
    &&<u(\cdot,t)>^\alpha\\
        &\leq& C\|\nabla u\|_{L^p}\nonumber\\
    &\leq&C(\|F\|_{L^p}
          +\|w\|_{L^p}+\|P-P(\tilde{\rho})\|_{L^p}
          )\nonumber\\
    &\leq&C\left(\|\rho\dot{u}\|_{L^2}^{\frac{p-2}{p}}
        (\|\nabla
        u\|_{L^2}^{\frac{2}{p}}+\|P-P(\tilde{\rho})\|_{L^2}^{\frac{2}{p}})
    +\|\nabla u\|_{L^2}+\|f\|_{L^2}+\|P-P(\tilde{\rho})\|_{L^2}
         \right)\nonumber\\
            &&+C\|\rho-\tilde{\rho}\|_{L^2}^\frac{2}{p}
            \|\rho-\tilde{\rho}\|_{L^\infty}^{1-\frac{2}{p}}\\
    &\leq&C\left(\|\rho\dot{u}\|_{L^2}^{\frac{p-2}{p}}
        (\|\nabla
        u\|_{L^2}^{\frac{2}{p}}+(C_0+C_f)^{\frac{1}{p}})
    +\|\nabla u\|_{L^2}+(C_0+C_f)^\frac{1}{p}
         \right).
    \end{eqnarray*}
\end{proof}

\begin{prop}\label{2VDD-P2.5}
 If $u_0\in H^1$, $(\rho,u)$
  is a smooth solution of (\ref{2VDD-E1.1})--(\ref{2VDD-E1.2}) as
  in Proposition \ref{2VDD-P2.1}, then we have
    \begin{equation}
      \sup_{0\leq t\leq T}\int
      |\nabla
      u|^2dx+\int^T_0\int\rho|\dot{u}|^2dxdt
        \leq C(1+M_q).\label{2VDD-E2.59}
    \end{equation}
\end{prop}
\begin{proof}
  Using a similar argument as that in the proof of
  (\ref{2VDD-E2.4}), we have
    $$
      \sup_{0\leq t\leq T}\int|\nabla
      u|^2dx+\int^T_0\int\rho|\dot{u}|^2dxdt
            \leq C(C_0+C_f+M_q)+C\int^T_0\int|\nabla u|^3dxds.
   $$
Without loss of generality, assume  that $T>1$. From
(\ref{2VDD-E2.22}) and (\ref{2VDD-E2.32}), we get
        $$
      \sup_{0\leq t\leq T}\int|\nabla
      u|^2dx+\int^T_0\int\rho|\dot{u}|^2dxdt
            \leq C(1+M_q)+C\int^1_0\int|\nabla u|^3dxds.
   $$
  From
(\ref{2VDD-E2.16-1}) and (\ref{2VDD-E2.16-4}), we have
    $$
\sup_{0\leq t\leq T}\int|\nabla
      u|^2dx+\int^T_0\int\rho|\dot{u}|^2dxdt
            \leq C(1+M_q)+C\int^1_0\int(|F|^3+|w|^3)dxds.
    $$
  From
(\ref{2VDD-E2.16-2})--(\ref{2VDD-E2.16-3}) and (\ref{2VDD-E2.16}),
we obtain
    \begin{eqnarray}
        && \int\left(|F|^3+|w|^3\right)dx\nonumber\\
                &\leq&
      C\left(\int |F|^2dx
      \right)\left(\int |\nabla F|^2dx
      \right)^{\frac{1}{2}}+C\left(\int |w|^2dx
      \right)\left(\int |\nabla w|^2dx
      \right)^{\frac{1}{2}}\nonumber\\
            &\leq&C\int(|\nabla u|^2+|\rho-\tilde{\rho}|^2)dx
            \left(\int\left(\rho|\dot{u}|^2+|f|^2
            \right)dx\right)^\frac{1}{2}.
    \end{eqnarray}
Thus, from Proposition \ref{2VDD-P2.1}, we have
     \begin{eqnarray*}
    &&\sup_{0\leq t\leq T}\int|\nabla
      u|^2dx+\int^T_0\int\rho|\dot{u}|^2dxdt\nonumber\\
            &\leq& C(1+M_q)+C\int^1_0\left(
            \|\nabla u\|_{L^2}^2+C_0+C_f
            \right)
            \left(\|\sqrt{\rho}\dot{u}\|_{L^2}+\|f\|_{L^2}
            \right)ds\nonumber\\
            &\leq& C(1+M_q)+\frac{1}{2}\int^T_0\int\rho|\dot{u}|^2dxdt+C\int^T_0
            \|\nabla u\|_{L^2}^4ds.
     \end{eqnarray*}
Using Gronwall's inequality, we obtain
    $$
    \sup_{0\leq t\leq T}\int|\nabla
      u|^2dx+\int^T_0\int\rho|\dot{u}|^2dxdt
      \leq C(1+M_q)e^{\int^T_0
            \|\nabla u\|_{L^2}^2ds}
      \leq C(1+M_q).
    $$
\end{proof}

\begin{prop}\label{2VDD-P2.6}
 If $u_0\in H^1$, $(\rho,u)$
  is a smooth solution of (\ref{2VDD-E1.1})--(\ref{2VDD-E1.2}) as
  in Proposition \ref{2VDD-P2.1}, then we have
    \begin{equation}
      \sup_{0<t\leq T}\sigma\int\rho|\dot{u}|^2dx
      +\int^T_0\int\sigma|\nabla \dot{u}|^2dxdt
        \leq C(1+M_q).\label{2VDD-E2.62}
    \end{equation}
\end{prop}
\begin{proof}
  Using a similar argument as that in the proof of
  (\ref{2VDD-E2.5}), from (\ref{2VDD-E2.59}), we have
    $$
      \sup_{0<t\leq T}\sigma\int\rho|\dot{u}|^2dx
      +\int^T_0\int\sigma|\nabla \dot{u}|^2dxdt
            \leq C(1+M_q)+C\int^T_0\int\sigma\left(|u|^4+|\nabla u|^4\right)dxds.
   $$
Without loss of generality, assume that $T>1$. From
(\ref{2VDD-E2.22}) and (\ref{2VDD-E2.31}), we get
        $$
   \sup_{0<t\leq T}\sigma\int\rho|\dot{u}|^2dx
      +\int^T_0\int\sigma|\nabla \dot{u}|^2dxdt
            \leq C(1+M_q)+C\int^1_0\int\sigma\left(|u|^4+|\nabla u|^4\right)dxds.
   $$
  From (\ref{2VDD-E2.1}),
(\ref{2VDD-E2.16-0})--(\ref{2VDD-E2.16-1}), (\ref{2VDD-E2.16-4}) and
(\ref{2VDD-E2.59}), we have
    $$
\sup_{0<t\leq T}\sigma\int\rho|\dot{u}|^2dx
      +\int^T_0\int\sigma|\nabla \dot{u}|^2dxdt
            \leq C(M_q)+C\int^1_0\int\sigma(|F|^4+|w|^4)dxds.
    $$
  From (\ref{2VDD-E2.1}),
(\ref{2VDD-E2.16-2}), (\ref{2VDD-E2.16}) and (\ref{2VDD-E2.59}),
we obtain
    \begin{eqnarray}
        && \int^1_0\int\sigma\left(|F|^4+|w|^4\right)dxds\nonumber\\
                &\leq&
      C\int^1_0\sigma\left(\int |F|^2dx
      \right)\left(\int |\nabla F|^2dx
      \right)ds+
      C\int^1_0\sigma\left(\int |w|^2dx
      \right)\left(\int |\nabla w|^2dx
      \right)ds\nonumber\\
            &\leq&C\int^1_0\sigma\left(\int(|\nabla
            u|^2+|\rho-\tilde{\rho}|^2)dx\right)
            \left(\int\left(\rho|\dot{u}|^2+|f|^2
            \right)dx\right)ds\nonumber\\
                &\leq&C(1+M_q)+\int^1_0\|\nabla u\|_{L^2}^2\sigma\int\rho|\dot{u}|^2dxds.
    \end{eqnarray}
Using Gronwall's inequality, we can finish the proof of this
proposition.
\end{proof}

\begin{lem} For any $p\in[2,\infty)$, we have
    \begin{equation}
     \|\dot{u}\|_{L^p}\leq C_p\|\sqrt{\rho}\dot{u}\|_{L^2}^{\frac{2}{p}}
     \|\nabla\dot{u}\|_{L^2}^{1-\frac{2}{p}}+C_p\|\nabla\dot{u}\|_{L^2}.\label{2VDD-E2.64-0}
    \end{equation}
\end{lem}
\begin{proof}
  Since
    $$
    \tilde{\rho}\int|\dot{u}|^2dx\leq \int\rho |\dot{u}|^2dx+
    \left(\int|\rho-\tilde{\rho}|^2
    dx    \right)^\frac{1}{2}
    \left(\int|\dot{u}|^4
    dx    \right)^\frac{1}{2},
    $$
applying (\ref{2VDD-E2.16}) and Proposition \ref{2VDD-P2.6}, we get
         $$
    \|\dot{u}\|_{L^2}^2\leq
    C\|\sqrt{\rho}\dot{u}\|_{L^2}^2+C\|\nabla \dot{u}\|_{L^2}^2.
    $$
From (\ref{2VDD-E2.16}), we can immediately obtain
(\ref{2VDD-E2.64-0}).
\end{proof}

\begin{lem}
  For any $q\in (0,2)$, we have
    \begin{equation}
      \int^T_0\int\sigma^{\frac{2+q}{2}}\rho|\dot{u}|^{2+q}dxds
      \leq C(1+M_q^{1+\frac{q}{2}}).\label{2VDD-E2.65}
    \end{equation}
\end{lem}
\begin{proof} Let $p=\frac{4}{2-q}$. Using H\"{o}lder's inequality, (\ref{2VDD-E2.62}) and
(\ref{2VDD-E2.64-0}), we have
    \begin{eqnarray*}
    &&\int^T_0\int\sigma^{\frac{2+q}{2}}\rho|\dot{u}|^{2+q}dxds\\
        &\leq&C\int^T_0\sigma^{\frac{2+q}{2}}\|\sqrt{\rho}\dot{u}\|_{L^2}^\frac{2(p-2-q)}{p-2}
        \|\dot{u}\|_{L^p}^{\frac{qp}{p-2}}ds\\
         &\leq&C\int^T_0\sigma^{\frac{2+q}{2}}\|\sqrt{\rho}\dot{u}\|_{L^2}^\frac{2(p-2-q)}{p-2}
        \left(\|\sqrt{\rho}\dot{u}\|_{L^2}^{\frac{2}{p}}
     \|\nabla\dot{u}\|_{L^2}^{1-\frac{2}{p}}+\|\nabla\dot{u}\|_{L^2}
        \right)^{\frac{qp}{p-2}}ds\\
            &\leq& C\left(\int^T_0\sigma\|\nabla \dot{u}\|_{L^2}^2dt
            \right)^\frac{q}{2}
            \left(\int^T_0\|\sqrt{\rho} \dot{u}\|_{L^2}^2dt
            \right)^\frac{2-q}{2}
            \left(\sup_{t\in[0,T]}\sigma\|\sqrt{\rho} \dot{u}\|_{L^2}^2
            \right)^\frac{q}{2}\\
        && +C\left(\int^T_0\sigma\|\nabla \dot{u}\|_{L^2}^2dt
            \right)^\frac{qp}{2(p-2)}
            \left(\int^T_0\|\sqrt{\rho} \dot{u}\|_{L^2}^2dt
            \right)^\frac{2p-4-qp}{2(p-2)}
            \left(\sup_{t\in[0,T]}\sigma\|\sqrt{\rho} \dot{u}\|_{L^2}^2
            \right)^\frac{q}{2}\\
                &\leq& C(1+M_q^{1+\frac{q}{2}}).
    \end{eqnarray*}
\end{proof}

\begin{prop}\label{2VDD-P2.7}
 If $u_0\in H^1$, $(\rho,u)$
  is a smooth solution of (\ref{2VDD-E1.1})--(\ref{2VDD-E1.2}) as
  in Proposition \ref{2VDD-P2.6}, $q\in (0,2)$ and
    \begin{equation}
      q^2<\frac{4\mu}{\lambda(\rho)+\mu},
     \ \forall\ \rho\in[0,\bar{\rho}],\label{2VDD-E2.64}
    \end{equation}
  then we have
    \begin{equation}
      \sup_{0<t\leq T}\sigma^{2+\frac{q}{2}}\int\rho|\dot{u}|^{2+q}dx
      +\int^T_0\int\sigma^{2+\frac{q}{2}}|\dot{u}|^q|\nabla \dot{u}|^2dxdt
        \leq C(M_q).\label{2VDD-E2.66}
    \end{equation}
\end{prop}
\begin{proof}
Using a similar argument as that in the proof of (\ref{2VDD-E2.11}),
we have
   \begin{eqnarray}
      &&\frac{1}{2+q}\sigma^{2+\frac{q}{2}}\int\rho|\dot{u}|^{2+q}dx\nonumber\\
            &=&\int^t_0\int\{\frac{4+q}{4+2q}\sigma^{1+\frac{q}{2}}
            \sigma'\rho|\dot{u}|^{2+q}-\sigma^{2+\frac{q}{2}}|\dot{u}|^q\dot{u}^j\left(
            \partial_j P_t+\mathrm{div}(\partial_j Pu)\right)\nonumber\\
                &&
            +\mu\sigma^{2+\frac{q}{2}}|\dot{u}|^q\dot{u}^j\left(\Delta u^j_t+\mathrm{div}(u\Delta
            u^j)\right)\nonumber\\
                &&+\sigma^{2+\frac{q}{2}}|\dot{u}|^q\dot{u}^j\left(
                \partial_j\partial_t((\lambda+\mu)\mathrm{div}u)
                +\mathrm{div}(u\partial_j((\lambda+\mu)\mathrm{div}u))\right)\nonumber\\
                    &&
                +\sigma^{2+\frac{q}{2}}|\dot{u}|^q\dot{u}^j\left((\rho f^j)_t+\mathrm{div}(u\rho f^j)
                \right)\}dxds\nonumber\\
                &:=&\sum^5_{i=1}H_i.\label{2VDD-E2.67}
    \end{eqnarray}
Using the integration by parts and H\"{o}lder's inequality, we have
    \begin{eqnarray}
      H_2&=&-\int^t_0\int\sigma^{2+\frac{q}{2}}|\dot{u}|^q\dot{u}^j\left(
            \partial_j P_t+\mathrm{div}(\partial_j
            Pu)\right)dxds\nonumber\\
                &=&\int^t_0\int\sigma^{2+\frac{q}{2}}|\dot{u}|^q\left(\partial_j\dot{u}^jP'\rho_t
                +\partial_k\dot{u}^j\partial_jP
                u^k\right)dxds\nonumber\\
                    &&+\int^t_0\int\sigma^{2+\frac{q}{2}}\left(\partial_j|\dot{u}|^q
                    \dot{u}^jP'\rho_t
                +\partial_k|\dot{u}|^q\dot{u}^j\partial_jP
                u^k\right)dxds\nonumber\\
      &=&\int^t_0\int\sigma^{2+\frac{q}{2}}\left(
        -|\dot{u}|^qP'\rho\mathrm{div}u\partial_j\dot{u}^j+\partial_k(|\dot{u}|^q\partial_j\dot{u}^ju^k)P
      -P\partial_j(|\dot{u}|^q\partial_k\dot{u}^ju^k)\right)dxds\nonumber\\
        &&+\int^t_0\int\sigma^{2+\frac{q}{2}}\left(
        -\partial_j|\dot{u}|^q\dot{u}^jP'\rho\mathrm{div}u
        +P\partial_k(\partial_j|\dot{u}|^q\dot{u}^ju^k)
                                -P\partial_j(\partial_k|\dot{u}|^q\dot{u}^ju^k)
                \right)dxds \nonumber\\
        &=&\int^t_0\int\sigma^{2+\frac{q}{2}}\left(
        -|\dot{u}|^{q}P'\partial_j\dot{u}^j\rho\mathrm{div}u
        +P\partial_k|\dot{u}|^q\partial_j\dot{u}^ju^k+P|\dot{u}|^q\partial_j\dot{u}^j\mathrm{div}u\right.\nonumber\\
            &&-P\partial_j|\dot{u}|^q\partial_k\dot{u}^ju^k-P|\dot{u}|^q\partial_k\dot{u}^j\partial_ju^k
            -\partial_j|\dot{u}|^q\dot{u}^jP'\rho\mathrm{div}u\nonumber\\
                &&\left.+P\partial_k(\dot{u}^ju^k)\partial_j|\dot{u}|^q
                -P\partial_j(\dot{u}^ju^k)\partial_k|\dot{u}|^q\right)dxds\nonumber\\
            &=&\int^t_0\int\sigma^{2+\frac{q}{2}}\big(
        -|\dot{u}|^{q}P'\partial_j\dot{u}^j\rho\mathrm{div}u
        +P|\dot{u}|^q\partial_j\dot{u}^j\mathrm{div}u-P|\dot{u}|^q\partial_k\dot{u}^j\partial_ju^k\nonumber\\
            &&
            -\partial_j|\dot{u}|^q\dot{u}^jP'\rho\mathrm{div}u
                +P\dot{u}^j\mathrm{div}u\partial_j|\dot{u}|^q
                -P\dot{u}^j\partial_ju^k\partial_k|\dot{u}|^q\big)dxds\nonumber\\
            &\leq&C\int^T_0\int\sigma^{2+\frac{q}{2}}|\nabla\dot{u}||\dot{u}|^q|\nabla
            u|dxds,
    \end{eqnarray}
        \begin{eqnarray}
          H_3&=&\int^t_0\int\mu\sigma^{2+\frac{q}{2}}|\dot{u}|^q\dot{u}^j
          \left(\Delta u^j_t+\mathrm{div}(u\Delta
            u^j)\right)dxds\nonumber\\
                &=&-\int^t_0\int\sigma^{2+\frac{q}{2}}\mu\left(
                \partial_i(|\dot{u}|^q\dot{u}^j)\partial_iu^j_t+\Delta u^j
                u\cdot\nabla (|\dot{u}|^q\dot{u}^j)\right)dxds\nonumber\\
            &=&-\int^t_0\int\sigma^{2+\frac{q}{2}}\mu\left(|\dot{u}|^q
            |\nabla \dot{u}|^2-\partial_i\dot{u}^j
            u^k\partial_k\partial_i
            u^j|\dot{u}|^q-\partial_i \dot{u}^j|\dot{u}|^q\partial_i u^k\partial_k
            u^j\right.\nonumber\\
                &&\left.+\partial_i|\dot{u}|^q\dot{u}^j\partial_iu^j_t
            +\Delta u^ju\cdot\nabla(|\dot{u}|^q\dot{u}^j)\right)dxds\nonumber\\
                    &=&-\int^t_0\int\sigma^{2+\frac{q}{2}}\mu\big(|\dot{u}|^q|\nabla \dot{u}|^2
                    +\partial_i\dot{u}^j\mathrm{div}
                    u\partial_iu^j|\dot{u}|^q+\partial_i|\dot{u}|^q\dot{u}^j\mathrm{div}u\partial_iu^j\nonumber\\
            &&-\partial_i \dot{u}^j|\dot{u}|^q\partial_i u^k\partial_k u^j
            +\partial_i|\dot{u}|^q\dot{u}^j\partial_i\dot{u}^j
            -\partial_i|\dot{u}|^q\dot{u}^j\partial_iu^k\partial_ku^j\nonumber\\
                &&-\partial_iu^j\partial_iu^k\partial_k|\dot{u}|^q\dot{u}^j
                -\partial_i u^j|\dot{u}|^q\partial_i
                    u^k\partial_k\dot{u}^j
                    \big)dxds\nonumber\\
            &\leq&-\int^t_0\int\sigma^{2+\frac{q}{2}}\mu|\dot{u}|^q|\nabla \dot{u}|^2dxds
            -\int^t_0\int\frac{q}{4}\sigma^{2+\frac{q}{2}}\mu|\dot{u}|^{q-2}|\nabla
            |\dot{u}|^2|^2dxds\nonumber\\
                &&
            +C\int^t_0\int \sigma^{2+\frac{q}{2}}|\nabla u|^2|\dot{u}|^q|\nabla \dot{u}|dxds,
        \end{eqnarray}
            \begin{eqnarray}
                H_4&=&\int^t_0\int\sigma^{2+\frac{q}{2}}|\dot{u}|^q
                \dot{u}^j\left(\partial_j\partial_t((\lambda+\mu)\mathrm{div}u)
                +\mathrm{div}(u\partial_j((\lambda+\mu)\mathrm{div}u))\right)dxds\nonumber\\
            &=&\int^t_0\int\{-\sigma^{2+\frac{q}{2}}\partial_j(|\dot{u}|^q
            \dot{u}^j)\left(\partial_t((\lambda+\mu)\mathrm{div}u)
                +\mathrm{div}(u(\lambda+\mu)\mathrm{div}u)\right)\nonumber\\
                    &&-\sigma^{2+\frac{q}{2}}|\dot{u}|^q\dot{u}^j\mathrm{div}
                    (\partial_ju(\lambda+\mu)\mathrm{div}u)\}dxds\nonumber\\
                        &=&\int^t_0\int-\sigma^{2+\frac{q}{2}}\partial_j
                        (|\dot{u}|^q\dot{u}^j)\big(\partial_t((\lambda+\mu)\mathrm{div}u)\nonumber\\
                            &&
                +u^k\lambda'\partial_k\rho\mathrm{div}u+(\lambda+\mu)u^k\partial_k\mathrm{div}u
                \big)dxds
                +O_5\nonumber\\
                                &=&\int^t_0\int-\sigma^{2+\frac{q}{2}}\partial_j
                                (|\dot{u}|^q\dot{u}^j)\left((\lambda+\mu)\frac{D}{Dt}\mathrm{div}u
                                +\lambda'\rho_t\mathrm{div}u
                +u^k\lambda'\partial_k\rho\mathrm{div}u\right)dxds+O_5\nonumber\\
                             &=&-\int^t_0\int\sigma^{2+\frac{q}{2}}|\dot{u}|^q\partial_j(\partial_tu^j
                             +u\cdot\nabla u^j)(\lambda+\mu)\frac{D}{Dt}
                             \mathrm{div}udxds+O_5\nonumber\\
                &&-\int^t_0\int\sigma^{2+\frac{q}{2}}\partial_j|\dot{u}|^q
                \dot{u}^j(\lambda+\mu)\frac{D}{Dt}
                             \mathrm{div}udxds\nonumber\\
            &\leq&-\int^t_0\int\sigma^{2+\frac{q}{2}}(\lambda+\mu)|\dot{u}|^q\left|\frac{D}{Dt}
                             \mathrm{div}u\right|^2dxds
                             +O_5\nonumber\\
            &&+\int^t_0\int q\sigma^{2+\frac{q}{2}}(\lambda+\mu)|\dot{u}|^q\left|\frac{D}{Dt}
                             \mathrm{div}u\right||\nabla\dot{u}|dxds,
            \end{eqnarray}
        \begin{eqnarray}
          H_5&=&\int^t_0\int \sigma^{2+\frac{q}{2}}|\dot{u}|^q\dot{u}^j[(\rho f^j)_t+\mathrm{div}(u\rho f^j)]
           dxds\nonumber\\
                &=&\int^t_0\int \sigma^{2+\frac{q}{2}}|\dot{u}|^q\dot{u}^j[\rho f^j_t+\rho u\cdot\nabla
                f^j]dxds\nonumber\\
                &\leq &C\int^t_0\int \sigma^{2+\frac{q}{2}}\left(\rho|\dot{u}|^{2+q}+ |f_t|^{2+q}
                +|u|^{4+2q}+|\nabla
                f|^{4+2q}\right)dxds,\label{2VDD-E2.71}
        \end{eqnarray}
where $O_5$ denotes any term dominated by
$C\int^t_0\int\sigma^{2+\frac{q}{2}}|\dot{u}|^q|\nabla u|^2(|\nabla
\dot{u}|+|\frac{D}{Dt}\mathrm{div}u|)dxds$, and $t\in[0,T]$. From
(\ref{2VDD-E2.16-0})--(\ref{2VDD-E2.16-1}),
(\ref{2VDD-E2.16-3})--(\ref{2VDD-E2.16-4}), (\ref{2VDD-E2.59}),
(\ref{2VDD-E2.62}), (\ref{2VDD-E2.64-0})--(\ref{2VDD-E2.64}) and
(\ref{2VDD-E2.67})--(\ref{2VDD-E2.71}), we have
    \begin{eqnarray*}
      && \sup_{0<t\leq T}\sigma^{2+\frac{q}{2}}\int\rho|\dot{u}|^{2+q}dx
      +\int^T_0\int\sigma^{2+\frac{q}{2}}|\dot{u}|^q|\nabla \dot{u}|^2dxdt
      \\
        &\leq& C(M_q)+C\int^T_0\int
        \sigma^{2+\frac{q}{2}}\left(|\dot{u}|^q|\nabla u|^4
        +|\dot{u}|^q|\nabla u|^2
                +|u|^{4+2q}\right)dxds\\
        &\leq&C(M_q)+C\left(\int^T_0
        \sigma\|\dot{u}\|_{L^{4+q}}^2ds\right)^\frac{q}{2}
        \left(\int^T_0\sigma^\frac{4}{2-q}\|\nabla u\|_{L^{4+q}}^\frac{8}{2-q}
                ds\right)^\frac{2-q}{2}+C\left(\int^T_0
        \sigma\|\dot{u}\|_{L^{4+q}}^2ds\right)^\frac{q}{2}\\
                &&\times
        \left(\int^T_0\sigma^\frac{4}{2-q}\|\nabla u\|_{L^{4+q}}^\frac{8}{2-q}
                ds\right)^\frac{q(2-q)}{4(2+q)}
                \left(\int^T_0\sigma^\frac{8+4q}{4-q^2}\|\nabla u\|_{L^{2}}^\frac{16}{(2-q)(4+q)}
                ds\right)^\frac{(2-q)(4+q)}{4(2+q)}\\
                    &&+C\int^T_0\sigma^{2+\frac{q}{2}}\left(\|\nabla
                    u\|_{L^2}^{2+2q}+\|\nabla
                    u\|_{L^2}^{4+2q}\right)dt\\
        &\leq&C(M_q)+C\int^T_0\sigma\left(\|\sqrt{\rho}\dot{u}\|_{L^2}^\frac{4}{4+q}\|\nabla
        \dot{u}\|_{L^2}^{\frac{2(2+q)}{4+q}}+\|\nabla
        \dot{u}\|_{L^2}^2
        \right)dt\\
            &&+C
        \int^T_0\sigma^\frac{4}{2-q}\left(\|F\|_{L^{4+q}}+\|w\|_{L^{4+q}}+\|\rho-\tilde{\rho}
        \|_{L^{4+q}}\right)^\frac{8}{2-q}
                ds\\
                         &\leq&C(M_q)+C
        \int^T_0\sigma^\frac{4}{2-q}\left(\|F\|_{L^{4+q}}+\|w\|_{L^{4+q}}
        \right)^\frac{8}{2-q}                ds+C
        \int^T_0\sigma^\frac{4}{2-q}\|\rho-\tilde{\rho}
        \|_{L^{4+q}}^{4+q}
                ds\\
         &\leq&C(M_q)+C
        \int^T_0\sigma^\frac{4}{2-q}\left(\|F\|_{L^2}^\frac{2}{4+q}
        \|\nabla F\|_{L^2}^\frac{2+q}{4+q}+\|w\|_{L^2}^\frac{2}{4+q}
        \|\nabla w\|_{L^2}^\frac{2+q}{4+q}
        \right)^\frac{8}{2-q}
                ds\\
                    &&+C
        \int^T_0\sigma^\frac{4}{2-q}\|F
        \|_{L^{4+q}}^{4+q}
                ds\\
        &\leq&C(M_q)+C
        \int^T_0\sigma^\frac{4}{2-q}(\|\nabla u\|_{L^2}+\|\rho-\tilde{\rho}\|_{L^2})^\frac{16}{(4+q)(2-q)}
        (\|\rho \dot{u}\|_{L^2}+\|f\|_{L^2})^\frac{8(2+q)}{(4+q)(2-q)}
                ds\\
        &&+C        \int^T_0\sigma^\frac{4}{2-q}(\|\nabla
        u\|_{L^2}+\|\rho-\tilde{\rho}\|_{L^2})^2
        (\|\rho \dot{u}\|_{L^2}+\|f\|_{L^2})^{2+q}
                ds\\
        &\leq&C(M_q).
    \end{eqnarray*}
\end{proof}

\begin{prop}\label{2VDD-P2.8}
 If $f\in L^\infty_tL^{2+q}_x$, $(\rho,u)$
  is a smooth solution of (\ref{2VDD-E1.1})--(\ref{2VDD-E1.2}) as
  in Proposition \ref{2VDD-P2.7}, then we have
    \begin{equation}
      \|F\|_{L^\infty}+\|w\|_{L^\infty}\leq C(\|\nabla u\|_{L^2}+\|\rho-\tilde{\rho}\|_{L^2})^{\frac{q}{2+2q}}
        (\|\rho\dot{u}\|_{L^{2+q}}+\|f\|_{L^{2+q}})^{\frac{2+q}{2+2q}}\label{2VDD-E2.72-1}
    \end{equation}
  and
    \begin{equation}
     \int^T_0(\|F\|_{L^\infty}+\|w\|_{L^\infty})ds
        \leq C(M_q)(C_0+C_f)^{\frac{q\theta}{4+4q}}(1+T).\label{2VDD-E2.72}
    \end{equation}
\end{prop}
\begin{proof}
  From (\ref{2VDD-E2.1}), (\ref{2VDD-E2.16-2}),
  (\ref{2VDD-E2.22}),
   (\ref{2VDD-E2.66}) and  the Galiardo-Nirenberg inequality, we have
    \begin{eqnarray*}
        &&    \|F\|_{L^\infty}\\
        &\leq& C\|F\|_{L^2}^{\frac{q}{2+2q}}
        \|\nabla F\|_{L^{2+q}}^{\frac{2+q}{2+2q}}\\
        &\leq& C(\|\nabla u\|_{L^2}+\|\rho-\tilde{\rho}\|_{L^2})^{\frac{q}{2+2q}}
        (\|\rho\dot{u}\|_{L^{2+q}}+\|f\|_{L^{2+q}})^{\frac{2+q}{2+2q}}.
    \end{eqnarray*}
  and
    \begin{eqnarray*}
        &&   \int^T_0 \|F\|_{L^\infty}ds\\
        &\leq& C(M_q)\int^T_0(\sigma^{-\frac{1}{2}}(C_0+C_f)^\frac{\theta}{2})^{\frac{q}{2+2q}}
        (\sigma^{-\frac{4+q}{4+2q}})^{\frac{2+q}{2+2q}}ds\\
        &\leq& C(M_q)(C_0+C_f)^{\frac{q\theta}{4+4q}}\int^T_0
        \sigma^{-\frac{2+q}{2+2q}}ds\leq C(M_q)(C_0+C_f)^{\frac{q\theta}{4+4q}}(1+T).
    \end{eqnarray*}
Similarly, we can obtain the same estimates for $w$.
\end{proof}

 Then, we derive a priori pointwise bounds for the density
$\rho$.

\begin{prop}\label{2VDD-P2.3}
  Given numbers $0<\underline{\rho}_2<\underline{\rho}_1<\tilde{\rho}
  <\bar{\rho}_1<\bar{\rho}_2<\bar{\rho}$, there is an
  $\varepsilon>0$ such that, if $(\rho,u)$
  is a smooth solution of (\ref{2VDD-E1.1})--(\ref{2VDD-E1.2}) with
  $C_0+C_f\leq\varepsilon$ and
  $0<\rho_0\leq\bar{\rho}_1$,
   then
        \begin{equation}
          0<\rho\leq\bar{\rho}_2,
          \ (x,t)\in\mathbb{R}^2\times[0,T].\label{2VDD-E2.37-0}
        \end{equation}
  Similarly, if $\rho_0\geq\underline{\rho}_1$ for all $x$, then
  $\rho\geq\underline{\rho}_2$ for all $x$ and $t$. Furthermore,
  the estimates in Propositions \ref{2VDD-P2.1}-\ref{2VDD-P2.8}  hold.
\end{prop}
\begin{proof}
At first, we prove that if (\ref{2VDD-E2.1-1}) and
(\ref{2VDD-E2.1-2}) hold, then estimate (\ref{2VDD-E2.37-0})
holds.

  We fix a curve $x(t)$ satisfying $\dot{x}= u(x(t),t)$ and $x(0)=x$. From
  (\ref{2VDD-E1.18}), we have
    \begin{equation}
 \frac{d}{dt}\Lambda(\rho(x(t),t))+P(\rho(x(t),t))-P(\tilde{\rho})=-F
 (x(t),t),\label{2VDD-E2.36}
    \end{equation}
where $\Lambda$ satisfies that $\Lambda(1)=0$ and
$\Lambda'(\rho)=\frac{2\mu+\lambda(\rho)}{\rho}$.

First for small time, we estimate the pointwise bounds of the
density as follows. From (\ref{2VDD-E2.1-1}) and (\ref{2VDD-E2.72}),
we have, for all $t\in[0,1]$,
    $$
    \left|\Lambda(\rho(x(t),t))-\Lambda(\rho_0(x))
    \right|
        \leq C(M_q)(C_0+C_f)^{\frac{q\theta}{4+4q}}+Ct.
    $$
When
    \begin{equation}
    2C(M_q)\varepsilon^{\frac{q\theta}{4+4q}}\leq
        \Lambda(\bar{\rho}_1+\frac{1}{3}(\bar{\rho}_2-\bar{\rho}_1))-\Lambda(\bar{\rho}_1),
    \end{equation}
we get
    $$
      \Lambda(\rho(x(t),t))\leq \Lambda(\bar{\rho}_1+\frac{1}{3}(\bar{\rho}_2-\bar{\rho}_1)),
      \ \ t\in[0,\tau],
    $$
and
    \begin{equation}
          \rho(x,t)\leq\bar{\rho}_1+\frac{1}{3}(\bar{\rho}_2-\bar{\rho}_1),
          \ (x,t)\in\mathbb{R}^2\times[0,\tau],
        \end{equation}
where
$\tau=\min\{1,\frac{1}{2C}(\Lambda(\bar{\rho}_1+\frac{1}{3}(\bar{\rho}_2-\bar{\rho}_1))-\Lambda(\bar{\rho}_1))$.
Since $\rho_0>0$, then we have
    $$
    \Lambda(\rho(x(t),t))
        \geq\Lambda(\rho_0(x))-C(M_q)(C_0+C_f)^{\frac{q\theta}{4+4q}}-C\tau
    >-\infty, \ t\in[0,\tau],
    $$
and
    $$
    \rho>0,\ (x,t)\in\mathbb{R}^2\times[0,\tau].
    $$
Similarly, if $\rho_0\geq\underline{\rho}_1$ and
     \begin{equation}
    2C(M_q)\varepsilon^{\frac{q\theta}{4+4q}}\leq \Lambda(\underline{\rho}_1)-\Lambda(\underline{\rho}_1-\frac{1}{3}(
        \underline{\rho}_1-\underline{\rho}_2)),
    \end{equation}
we get
       \begin{equation}
          \rho\geq\underline{\rho}_1-\frac{1}{3}(
        \underline{\rho}_1-\underline{\rho}_2),
          \ (x,t)\in\mathbb{R}^2\times[0,\tau_1].
        \end{equation}
where $\tau_1=\min\{\tau,
\frac{1}{2C}(\Lambda(\underline{\rho}_1)-\Lambda(\underline{\rho}_1-\frac{1}{3}(
        \underline{\rho}_1-\underline{\rho}_2)))\}$.

Then, for large time $t\geq\tau$, we estimate the pointwise bounds
of density as follows. From (\ref{2VDD-E2.1}), (\ref{2VDD-E2.22}),
(\ref{2VDD-E2.66}),
 (\ref{2VDD-E2.72-1}) and (\ref{2VDD-E2.36}), we have
        \begin{equation}
      \frac{d \Lambda(\rho(x(t),t))}{dt}+P(\rho(x(t),t))-P(\tilde{\rho})=O_5(t),\label{2VDD-E2.55}
    \end{equation}
where
    $$
    |O_5(t)|\leq  C(\tau,M_q)(C_0+C_f)^{\frac{q\theta}{4+4q}},
    \ t\geq\tau.
    $$
Now, we apply a standard maximum principle argument to estimate
the upper bounds of density. Let
    $$
    t_0=\max\{t\in(\tau,T]|\Lambda(\rho(x(s),s))\leq\Lambda(
    \bar{\rho}_2),\
    \textrm{ for all }s\in[0,t]
    \}.
    $$
If $t_0<T$, we have
    $$
    \Lambda(\rho(x(t_0),t_0))=\Lambda(\bar{\rho}_2),
    $$
     $$
   \left. \frac{d\Lambda(\rho(x(t),t))}{dt}\right|_{t=t_0}\geq0,
    $$
and
    $$
    \rho(x(t_0),t_0)=\bar{\rho}_2.
    $$
From (\ref{2VDD-E2.55}), we have
    $$
    O_5(t_0)\geq P(\bar{\rho}_2)-P(\tilde{\rho}).
    $$
On the other hand, when
    \begin{equation}
    C(\tau,M_q)        \varepsilon^{\frac{q\theta}{4+4q}}<
        P(\bar{\rho}_2)-P(\tilde{\rho}),
    \end{equation}
we have
     $$
    O_5(t_0)< P(\bar{\rho}_2)-P(\tilde{\rho}).
    $$
It is a contradiction. Thus, we have $t_0=T$ and
        \begin{equation}
          \rho\leq\bar{\rho}_2,
          \ (x,t)\in\mathbb{R}^2\times[0,T].
        \end{equation}
Similarly, we can obtain the lower bound of the density.

Using the classical continuation method, (\ref{2VDD-E2.22}) and
(\ref{2VDD-E2.37-0}), we can finish the proof of this proposition.
\end{proof}

Then, we can prove the global existence of smooth solutions to
(\ref{2VDD-E1.1})--(\ref{2VDD-E1.2}).

\begin{prop}\label{2VDD-P2.4}
 Assume that $\rho_0-\tilde{\rho}\in W^{1,p}\cap C^{1+\alpha}$,
$u_0\in
  H^2\cap C^{2+\alpha}$, $p>2$, $\alpha\in(0,1)$,
  $\rho_0(x)\geq \underline{\rho}$ with some $\underline{\rho}>0$ for all $x\in\mathbb{R}^2$, $P,\lambda\in
  C^\infty([0,\bar{\rho}])$ and  $f\in
  C^\infty([0,\infty);C^\infty)$. Under the assumptions of Theorem \ref{2VDD-T1.1},
  then there exists a solution $(\rho,u)\in C^{1+\alpha,1+\alpha/2}\times
  C^{2+\alpha,1+\alpha/2}(\mathbb{R}^2\times[0,T])$ satisfying
  (\ref{2VDD-E1.1})--(\ref{2VDD-E1.2}) and for which the
  bound estimates of Propositions \ref{2VDD-P2.1}-\ref{2VDD-P2.3}
  hold,
   for all $T>0$.
\end{prop}
\begin{proof}
  Using similar arguments as that in \cite{Vaigant} and in the proof of Proposition 3.2 in
  \cite{Hoff1995}, one can obtain this proposition.
\end{proof}

\section{Proof of Theorem \ref{2VDD-T1.1}}\label{2VDD-S3}

Let $j_\delta(x)$ be a standard mollifying kernel of width $\delta$.
Define the approximate initial data $(\rho_0^\delta,u_0^\delta)$ by
    $$
    \rho_0^\delta=j_\delta*\rho_0+\delta,\
        u_0^\delta=j_\delta*u_0.
    $$
Assuming that similar smooth approximations have been constructed
for functions $P$, $f$ and $\lambda$, we may then apply Proposition
\ref{2VDD-P2.4} to obtain a global smooth solution
$(\rho^\delta,u^\delta)$ of (\ref{2VDD-E1.1})--(\ref{2VDD-E1.2})
with the initial data $(\rho_0^\delta,u_0^\delta)$, satisfying the
bound estimates of Propositions \ref{2VDD-P2.1}-\ref{2VDD-P2.8} hold
with constants independent of $\delta$.

First, we obtain the strong limit of $\{u^\delta\}$. From
(\ref{2VDD-E2.22}) and (\ref{2VDD-E2.35}), we have
    \begin{equation}
      <u^\delta(\cdot,t)>^\alpha\leq C(\tau),\
     \ t\geq\tau>0,\ \alpha\in(0,1).\label{2VDD-E3.1}
    \end{equation}
 From (\ref{2VDD-E3.1}), we have
    $$
    \left|u^\delta(x,t)-\frac{1}{|B_R(x)|}\int_{B_R(x)}u^\delta(y,t)dy
    \right|\leq C(\tau)R^\alpha,\ t\geq\tau>0.
    $$
Taking $R=1$, from (\ref{2VDD-E2.16-0}) and (\ref{2VDD-E2.22}), we
have
    \begin{equation}
      \|u^\delta\|_{L^\infty(\mathbb{R}^2\times[\tau,\infty))}\leq
      C(\tau).\label{2VDD-E3.2}
    \end{equation}
Then, we need only to derive a modulus of H\"{o}lder continuity in
time.
      For all $t_2\geq
t_1\geq\tau$, from (\ref{2VDD-E2.1}), (\ref{2VDD-E2.22}),
(\ref{2VDD-E2.64-0})
 and (\ref{2VDD-E3.2}), we have
    \begin{eqnarray*}
      &&|u^\delta(x,t_2)-u^\delta(x,t_1)|\nonumber\\
            &\leq&\frac{1}{|B_R(x)|}\int^{t_2}_{t_1}\int_{B_R(x)}
            |u^\delta_t(y,s)|dyds+C(\tau)R^\alpha\nonumber\\
      &\leq&CR^{-1}|t_2-t_1|^\frac{1}{2}\left(\int^{t_2}_{t_1}\int
            |u^\delta_t|^2dyds\right)^\frac{1}{2}+C(\tau)R^\alpha\nonumber\\
      &\leq&CR^{-1}|t_2-t_1|^\frac{1}{2}\left(\int^{t_2}_{t_1}\int
            |\dot{u}^\delta|^2+|u^\delta\cdot\nabla u^\delta|^2dyds\right)^\frac{1}{2}+C(\tau)R^\alpha\nonumber\\
      &\leq&C(\tau)(R^{-1}|t_2-t_1|^\frac{1}{2}+R^\alpha).
        \end{eqnarray*}
Choosing $R=|t_2-t_1|^{\frac{1}{2+2\alpha}}$, we have
   \begin{equation}
    <u^\delta>^{\alpha,\frac{\alpha}{2+2\alpha}}_{\mathbb{R}^2\times[\tau,\infty)}\leq
    C(\tau),\ \tau>0.\label{2VDD-E3.3}
    \end{equation}
From the Ascoli-Arzela theorem, we have (extract a subsequence)
    \begin{equation}
      u^\delta\rightarrow u,
      \ \textrm{uniformly on compact sets in
      }\mathbb{R}^2\times(0,\infty).\label{2VDD-E3.4}
    \end{equation}

Second, we obtain the strong limits of $\{F^\delta\}$ and
$\{w^\delta\}$. From (\ref{2VDD-E2.16-2})--(\ref{2VDD-E2.16-3}),
(\ref{2VDD-E2.22}) and (\ref{2VDD-E2.66}), using similar arguments
as that in the proof of (\ref{2VDD-E3.1})--(\ref{2VDD-E3.2}), we
have
    \begin{equation}
      <F^\delta(\cdot,t)>^{\alpha'}+\|F^\delta\|_{L^\infty(\mathbb{R}^2\times[\tau,T])}
      + <w^\delta(\cdot,t)>^{\alpha'}+\|w^\delta\|_{L^\infty(\mathbb{R}^2\times[\tau,T])}
      \leq C(\tau,T),\label{2VDD-E3.5}
    \end{equation}
where $0<\tau\leq t\leq T$ and ${\alpha'}\in(0,\frac{q}{2+q}]$.
The simple computation implies that
    \begin{eqnarray}
      F^\delta_t&=&
      \rho^\delta(2\mu+\lambda(\rho^\delta))\left(
      F^\delta\frac{d}{ds}\left.\left(\frac{1}{2\mu+\lambda(s)}\right)\right|_{s=\rho^\delta}
      +\frac{d}{ds}\left.\left(\frac{P(s)-P(\tilde{\rho})}{2\mu+\lambda(s)}\right)\right|_{s=\rho^\delta}
      \right)\mathrm{div}u^\delta\nonumber\\
        &&-u^\delta\cdot\nabla F^\delta+(2\mu+\lambda(\rho^\delta))\mathrm{div}\dot{u}^\delta-(2\mu+\lambda(\rho^\delta)
        )\partial_iu^\delta_j\partial_ju^\delta_i
    \end{eqnarray}
and
    \begin{equation}
    (w^\delta)^{k,j}_t=-u^\delta\cdot\nabla (w^\delta)^{k,j}
    +\partial_j \dot{u}^\delta_k
    -\partial_k \dot{u}^\delta_j
    -\partial_ju^\delta_i\partial_iu^\delta_k
    +\partial_ku^\delta_i\partial_iu^\delta_j.
    \end{equation}
Then, from (\ref{2VDD-E2.16-2}), (\ref{2VDD-E2.22}),
(\ref{2VDD-E2.31}), (\ref{2VDD-E3.2}) and (\ref{2VDD-E3.5}), we have
    $$
    \|F^\delta_t\|_{L^2(\mathbb{R}^2\times[\tau,T])}
    +\|w^\delta_t\|_{L^2(\mathbb{R}^2\times[\tau,T])}\leq
    C(\tau,T),\ T>\tau>0.
    $$
Using a similar argument as that in the proof of
(\ref{2VDD-E3.3}), we obtain
    \begin{equation}
    <F^\delta>^{{\alpha'},\frac{{\alpha'}}{2+2{\alpha'}}}_{\mathbb{R}^2\times[\tau,T]}+
     <w^\delta>^{{\alpha'},\frac{{\alpha'}}{2+2{\alpha'}}}_{\mathbb{R}^2\times[\tau,T]}\leq
    C(\tau,T),\ T>\tau>0.
    \end{equation}
and (extract a subsequence)
        \begin{equation}
      F^\delta\rightarrow F,
      \  w^\delta\rightarrow w,
      \ \textrm{uniformly on compact sets in
      }\mathbb{R}^2\times(0,\infty).
    \end{equation}

    Third, we obtain the strong limit of $\{\rho^\delta\}$. From
    (\ref{2VDD-E2.37-0}), we get (extract a subsequence)
    $$
    \rho^\delta\overset{*}{\rightharpoonup}\rho,
    \ \textrm{ weak-* in    }\ L^\infty(\mathbb{R}^2).
    $$
Let $\Phi(s)$ be an arbitrary continuous function on
$[0,\bar{\rho}]$. Then, we have that (extract a subsequence)
$\Phi(\rho^\delta)$ converges weak-$*$ in
$L^\infty(\mathbb{R}^2)$. Denote the weak-$*$ limit by
$\bar{\Phi}$:
    $$
     \Phi(\rho^\delta)\overset{*}{\rightharpoonup}\bar{\Phi},
    \ \textrm{ weak-* in    }\ L^\infty(\mathbb{R}^2).
    $$
From the definition of $F$, we have
    \begin{equation}
      \mathrm{div}u=\bar{\nu}F+\overline{P_0},
    \end{equation}
where
    $$
    \nu(\rho)=\frac{1}{2\mu+\lambda(\rho)},
    \ P_0(\rho)=\nu(\rho)(P(\rho)-P(\tilde{\rho})).
    $$
From (\ref{2VDD-E1.1}), we have
    $$
    \partial_t\overline{\rho\ln\rho}+\mathrm{div}(\overline{\rho\ln\rho}
    u)+F\overline{\rho\nu}+\overline{\rho P_0}=0
    $$
and
    $$
    \partial_t(\rho\ln\rho)+\mathrm{div}({\rho\ln\rho}
    u)+F\rho\overline{\nu}+\rho\overline{ P_0}=0.
    $$
Letting $\Psi=\overline{\rho\ln\rho}-{\rho\ln\rho}\geq0$, we
obtain
    \begin{equation}
      \partial_t\Psi+ +\mathrm{div}(\Psi
    u)+F(\overline{\rho\nu}-\rho\nu)
    +F\rho(\nu-\bar{\nu})+\overline{\rho P_0}-\rho\overline{
    P_0}=0.\label{2VDD-E3.11}
    \end{equation}
with the initial condition $\Psi|_{t=0}=0$ almost everywhere in
$\mathbb{R}^2$. Let $\phi(s)=s\ln s$. Since
    $$
    \phi''(s)=\frac{1}{s}\geq
\frac{1}{\bar{\rho}},
    \ s\in[0,\bar{\rho}],
    $$
we get
    $$
    \phi(\rho^\delta)-\phi(\rho)
    =\phi'(\rho)(\rho^\delta-\rho)
    +\frac{1}{2}\phi''(\rho+\xi(\rho^\delta-\rho))(\rho^\delta-\rho)^2,
    \ \xi\in[0,1],
    $$
and
    \begin{equation}
      \overline{\lim_{\delta\rightarrow0}}\|\rho^\delta-\rho\|_{L^2}^2
      \leq C\|\Psi\|_{L^1}.
    \end{equation}
Similarly, every function $f\in C^2([0,\bar{\rho}])$ satisfies
    \begin{equation}
      \left|\int g(\bar{f}-f(\rho))dx
      \right|\leq C\int |g|\Psi dx,
    \end{equation}
where $g$ is any function such that the integrations exist. Then,
when $\nu\in C^2([0,\bar{\rho}])$, we have
    \begin{equation}
      \left|\int F(\overline{\rho\nu}-\rho\nu)dx
      \right|\leq C\int |F|\Psi dx
    \end{equation}
and
    \begin{equation}
      \left|\int F\rho(\bar{\nu}-\nu)dx
      \right|\leq C\int |F|\Psi dx.
    \end{equation}
When $P_0\in C^2([0,\bar{\rho}])$, we have
    \begin{equation}
    \left|\int (\overline{\rho P_0}-\rho \overline{P_0})dx
      \right|\leq    \left|\int (\overline{\rho P_0}-\rho P_0)dx
      \right|+\left|\int \rho(\overline{P_0}-P_0)dx
      \right|\leq C\int \Psi dx.
    \end{equation}
When $P_0$ is monotone function on $[0,\bar{\rho}]$, using the
Lemma 5 in \cite{Vaigant}, we have
     \begin{equation}
     \overline{\rho P_0}\geq \rho \overline{P_0}.\label{2VDD-E3.18}
    \end{equation}
From (\ref{2VDD-E3.11})--(\ref{2VDD-E3.18}), we obtain
    $$
    \int\Psi dx\leq \int^t_0\int(1+ |F|)\Psi dxds.
    $$
Using (\ref{2VDD-E2.72}) and Gronwall's inequality, we get
    $$
    \Psi=0,\ (t,x)\in[0,T]\times\mathbb{R}^2,
    $$
and  (extract a subsequence)
    $$
    \rho^\delta-\tilde{\rho}\rightarrow \rho-\tilde{\rho},
    \ \textrm{ strongly in } L^k(\mathbb{R}^2\times[0,\infty)),
    $$
for all $k\in[2,\infty)$.

Thus, It is easy to show that the limit function $(\rho,u)$ are
indeed a weak solution of the system
(\ref{2VDD-E1.1})--(\ref{2VDD-E1.2}). This finishes the proof of
Theorem \ref{2VDD-T1.1}. {\hfill $\square$\medskip}

\section{Lagrangean Structure}\label{2VDD-S4}
\noindent\textbf{Proof of Theorem \ref{2VDD-T1.2}.}

\textbf{(1).} Here, we consider the case $t_0=0$. The proof of the
case $t>0$ is similar, and omit the details. We denote that
$X(t,x_0)=X(t;x_0,0)$.

To prove the existence of the integral curve $X(\cdot,x_0)$, we
first assume that $X^\delta(\cdot,x_0)$ is the corresponding
integral curve of $u^\delta$ with initial point $x_0\in
\mathbb{R}^2$,
    \begin{equation}
      X^\delta(t,x_0)=x_0+\int^t_0 u^\delta(X^\delta(s,x_0),s)ds,
      \ s\in[0,\infty).\label{2VDD-E4.1}
    \end{equation}

From (\ref{2VDD-E2.16-0}), (\ref{2VDD-E2.35})--(\ref{2VDD-E2.59})
and (\ref{2VDD-E2.62}), using a similar argument as that in the
proof of (\ref{2VDD-E3.2}), we have
 \begin{equation}
    <u^\delta(\cdot,t)>^\alpha\leq C+C\sigma^{-\frac{\alpha}{2}},\
    \alpha\in(0,1),
    \end{equation}
    \begin{equation}
      \|u^\delta(t,\cdot)\|_{L^\infty}\leq
      C+C\sigma^{-\frac{\alpha}{2}},\label{2VDD-E4.3}
    \end{equation}
and
      \begin{equation}
      \int^T_{T_1}\|u^\delta(t,\cdot)\|_{L^\infty}dt\leq
      C(T-T_1+T^{1-\frac{\alpha}{2}}-T^{1-\frac{\alpha}{2}}_1),
      \ 0\leq T_1<T<\infty.\label{2VDD-E4.4}
    \end{equation}
This implies that $X^\delta(\cdot,x_0)$ is H\"{o}lder continuous on
$[0,\infty)$, uniformly in $\delta$. Therefore, there is a
subsequence $X^{\delta_j}(\cdot,x_0)$ and a H\"{o}lder continuous
map $X(\cdot,x_0)$ such that $X^{\delta_j}(\cdot,x_0)\rightarrow
X(\cdot,x_0)$ uniformly on compacts sets in $[0,\infty)$. From this
uniform convergence and (\ref{2VDD-E3.4}), we have that
$X(\cdot,x_0)$ satisfies (\ref{2VDD-E1.26}).

Next we prove the uniqueness and continuous dependence for integral
curves of $u$. Thus, let $X_1(\cdot,y_1)$ and $X_2(\cdot,y_2)$ be
any two integral curves of $u$  with respective initial points
$y_1,y_2\in\mathbb{R}^2$ and define
    $$
    g(t)=\frac{|u(X_2(t,y_2),t)-u(X_1(t,y_1),t)|}{m(|X_2(t,y_2)-X_1(t,y_1)|)},
    $$
and
    $$
    g^\delta(t)=\frac{|u^\delta(X_2(t,y_2),t)-u^\delta(X_1(t,y_1),t)|}{m(|X_2(t,y_2)-X_1(t,y_1)|)},
    $$
for $t\in[0,\infty)$, where
    $$
    m(x)=\left\{
    \begin{array}{ll}
    x(1-\ln x), &0<x\leq 1,\\
        x,&1\leq x<\infty.
    \end{array}\right.
    $$
From (\ref{2VDD-E1.26}), we get
    \begin{equation}
      |X_2(t,y_2)-X_1(t,y_1)|\leq |y_2-y_1|+\int^t_0g(s)
      m(|X_2(s,y_2)-X_1(s,y_1)|)ds.\label{2VDD-E4.5}
    \end{equation}

Denote by $<v>_{LL}$ the log-Lipschtz seminorm of a given vector
field $v$ on $\mathbb{R}^2$:
    $$
    <v>_{LL}=\supetage{x,y\in\mathbb{R}^2}{x\neq
    y}\frac{|v(x)-v(y)|}{m(|x-y|)}.
    $$
From (\ref{2VDD-E2.24}) and Proposition 2.3.7 in \cite{Chemin}, we
have
    $$
    g^\delta(t)\leq <u^\delta(\cdot,t)>_{LL}\leq C\|u^\delta\|_{B^1_{\infty,\infty}}\leq C(
    \|u^\delta\|_{L^{2}}+\|F^\delta\|_{L^{\infty}}+\|\rho^\delta-\tilde{\rho}\|_{L^\infty}
    +\|w^\delta\|_{L^\infty}).
    $$
 From (\ref{2VDD-E1.15}), (\ref{2VDD-E2.1}), (\ref{2VDD-E2.16-0})  and (\ref{2VDD-E2.72}), we obtain
    $$
  \int^T_0  g^\delta dt\leq C(1+T).
    $$
Using (\ref{2VDD-E3.4})  and Fatou's lemma, we have
    $$
    \int^t_0g(s)ds\leq
    \liminf_{\delta\rightarrow0}\int^t_0g^\delta(s)ds\leq C(1+T).
    $$
Using Gronwall inequality in (\ref{2VDD-E4.5}), we get
    \begin{eqnarray}
      |X_2(t,y_2)-X_1(t,y_1)|&\leq&
      \exp(1-e^{-\int^t_0gds})|y_2-y_1|^{\exp(-\int^t_0gds)}\nonumber\\
        &\leq&
        \exp(1-e^{-C(1+T)})|y_2-y_1|^{\exp(-C(1+T))}.\label{2VDD-E4.6}
    \end{eqnarray}
Thus, if $y_1=y_2$, then $X_1=X_2$. There is therefore exactly one
integral curve originating from a given point at time $t=0$. From
this uniqueness, we obtain that the convergence
$X^\delta(t,y_1)\rightarrow X(t,y_1)$ uniformly on compact sets in
$[0,\infty)$ for entire sequence $\delta\rightarrow0$, independently
of $y_1$.

\textbf{(2).} We prove the injection of $X$ at first. Suppose that
$X(t,y_1)=X(t,y_2)$ for some $t>0$ and $y_1,y_2\in\mathbb{R}^2$.
Then for any $s\in[0,t)$,
    $$
    |X(s,y_1)-X(s,y_2)|\leq
    \int^t_s<u(\tau,\cdot)>_{LL}m(|X(\tau,y_1)-X(\tau,y_2)|)d\tau.
    $$
Using a similar argument as that in the proof of (\ref{2VDD-E4.6}),
we have that $X(s,y_1)=X(s,y_2)$ for all $s\in[0,t ]$.

Next we show that $X(t,\cdot)|_{\Omega}$ is an open mapping. Let
$A$ be an open subset of $\Omega$, $y_1\in A$ and
$B_{r_1}(y_1)\subset A$, $0\leq s<t$, $z_1=X(t,y_1)$. From
(\ref{2VDD-E1.26}), using a similar argument as that in the proof
of (\ref{2VDD-E4.6}), we have
    $$
    |X(s,y_1)-X(s;z,t)|\leq
    \exp(1-e^{-C(1+t)})|z_1-z|^{\exp(-C(1+t))}.
    $$
Thus, there is a sufficient small constant $r_2$ such that, if $z\in
B_{r_2}(z_1)$, then
    \begin{equation}
      |y_1-X(0;z,t)|< r_1.\label{2VDD-E4.8-0}
    \end{equation}
Thus $X(0;z,t)\in B_{r_1}(y_1)\subset A$ if $z\in B_{r_2}(z_1)$.
From the uniqueness of the integral curves of $u$, we obtain that
$z=X(t,X(0;z,t))\in X(t,\cdot)A$.  Therefore, $B_{r_2}(z_1)\subset
X(t,\cdot)A$, and $X(t,\cdot)|_{\Omega}$ is an open mapping.

\textbf{(3).} Using a similar argument as that in the proof of
(\ref{2VDD-E4.6}), we have
 \begin{eqnarray*}
      |X(t_2,y_2)-X(t_2,y_1)|&\leq&
      \exp(1-e^{-\int^{t_2}_{t_1}gds})|X(t_1,y_2)-X(t_1,y_1)|^{\exp(-\int^{t_2}_{t_1}gds)}\nonumber\\
        &\leq&
        \exp(1-e^{-C(1+T)})|X(t_1,y_2)-X(t_1,y_1)|^{\exp(-C(1+T))},
    \end{eqnarray*}
which proves part (3).

\textbf{(4).} Part (4) of Theorem \ref{2VDD-T1.2} is an immediate
consequence of part (3).{\hfill $\square$\medskip}

To prove Theorem \ref{2VDD-T1.3}, we need the following lemma.
\begin{lem}\label{2VDD-L4.1}
  Given $x\in X(t,\cdot)V$, say $x=X(t,y)$ with $y\in V$, there is
  a sequence $\delta_j\rightarrow 0$, which may depend on $x$,
  such that $X^{\delta_j}(t,\cdot)^{-1}(x)$ tends to $y$ as
  $\delta_j\rightarrow0$.
\end{lem}
\begin{proof}
  Using the argument as that in the proof of Part (1) of Theorem
  \ref{2VDD-T1.2}, we have that integral curves $X^\delta(s;x,t)$
  of the approximate velocity field $u^\delta$, defined by
    $$
    X^\delta(s;x,t)=x-\int^t_s u^\delta(X^\delta(\tau;x,t),\tau)d
    \tau,
    $$
  are H\"{o}lder continuous in $s\in [0,t]$, uniformly with respect to
  $\delta$. Therefore, there is a sequence $\delta_j\rightarrow0$
  and a map $\tilde{X}\in C([0,t];\mathbb{R}^2)$ such that
  $X^{\delta_j}(\cdot;x,t)$ converges uniformly to
  $\tilde{X}(\cdot)$, which satisfies
    $$
    \tilde{X}(s)=x-\int^t_su(\tilde{X}(\tau),\tau)d\tau,
    \ 0\leq s\leq t.
    $$
  From the uniqueness of integral curves proved in Part (1) of
  Theorem \ref{2VDD-T1.2}, we have that $\tilde{X}(t)=X(t,y)$.
  Taking $s=0$, from (\ref{2VDD-E3.4}) and (\ref{2VDD-E4.3}), we get that
  $y^{\delta_j}=X^{\delta_j}(t,\cdot)^{-1}(x)$ converges to $y$ as
  $\delta_j$ tends to zero.
\end{proof}

\noindent\textbf{Proof of Theorem \ref{2VDD-T1.3}.}

Applying a standard maximum principle to the mass equation, using
a similar argument as that in the proof of Proposition
\ref{2VDD-P2.3}, we have
     $$
    \rho^\delta(X^\delta(t,y),t)\geq \underline{\rho}^-,
    $$
for any $y\in V$, and all $\delta$ sufficiently small. Let
$x=X(t,y)\in V^t$ and $y^\delta=X^{\delta}(t,\cdot)^{-1}(x)$. From
Lemma \ref{2VDD-L4.1}, we have that there is a sequence $\delta_j$
tending to zero such that $y^{\delta_j}$ tends to $y\in V$. Then,
for all sufficient small $\delta_j$, we get
    $$
  y^{\delta_j}\in V\ \textrm{ and }\ \rho^{\delta_j}(x,t)
   =\rho^{\delta_j}(X^{\delta_j}(t,y^{\delta_j}),t)\geq \underline{\rho}^-.
    $$
From the convergence $\rho^\delta\rightarrow\rho$ (which holds for
a.a. $x$), we obtain that
      $$
   \rho(x,t)=\rho(X(t,y),t)\geq \underline{\rho}^-,
    $$
for all $y\in V$. From Part (2) of Theorem \ref{2VDD-T1.2}, we can
finish the proof of this theorem.    {\hfill $\square$\medskip}

\noindent\textbf{Proof of Theorem \ref{2VDD-T1.4}.}

For any $y\in U$, there is a sufficient small $\delta_0$ such that
    $$
    \mathrm{dist}(y,\partial U)\geq 2\delta_0.
    $$
Let $U_{\delta_0}=\{x\in U|\ \mathrm{dist}(x,\partial U)\geq
\delta_0\}$. Then, we have
    $$
        \rho^\delta_0|_{U_{\delta_0}}=\delta,\ \forall\ \delta\leq \delta_0.
    $$
From (\ref{2VDD-E2.72}) and (\ref{2VDD-E2.36}), we have
     \begin{eqnarray*}
 \Lambda(\rho^\delta(X^\delta(t,z),t))&\leq&
 \Lambda(\rho^\delta_0(z))-\int^t_0P(\rho^\delta(X^\delta(s,z),s))ds+tP(\tilde{\rho})+\int^t_0\|F^\delta(\cdot,s)
 \|_{L^\infty}ds\\
    &\leq& \Lambda(\delta)+C(T),
    \end{eqnarray*}
  for all $t\in[0,T]$,  $\delta\leq\delta_0$ and all $z\in U_{\delta_0}$.
 Since $\Lambda(C\delta)-\Lambda(\delta)\geq 2\mu(\ln(C\delta)-\ln\delta)=2\mu\ln
 C$, then we have
    $$
    \Lambda(\rho^\delta(X^\delta(t,z),t))\leq \Lambda(\delta)+\Lambda(C(T)\delta)-\Lambda(\delta)
    =\Lambda(C(T)\delta),
       $$
 and
    $$
    \rho^\delta(X^\delta(t,z),t)\leq C(T)\delta\leq C(T)\delta_0,
    $$
 for all $t\in[0,T]$, $\delta\leq\delta_0$ and all $z\in
 U_{\delta_0}$.  Let
$x=X(t,y)\in U^t$ and $y^\delta=X^{\delta}(t,\cdot)^{-1}(x)$. From
Lemma \ref{2VDD-L4.1}, we have that there is a sequence $\delta_j$
tending to zero such that $y^{\delta_j}$ tends to $y\in
U_{\delta_0}$. Then, for all sufficient small $\delta_j$, we get
    $$
  y^{\delta_j}\in U_{\delta_0}\textrm{ and }  \rho^{\delta_j}
  (x,t)=\rho^\delta(X^{\delta_j}(t,y^{\delta_j}),t)\leq C(T)\delta_0,
    $$
   for all $t\in[0,T]$. From the convergence $\rho^\delta\rightarrow\rho$ (which holds for
a.a. $x$), we obtain that
      $$
   \rho(x,t)=\rho(X(t,y),t)\leq C(T)\delta_0,
    $$
   for all $t\in[0,T]$. Letting $\delta_0\rightarrow0$, we get
   that $\rho(X(t,y),t)=0$   for all $t\in[0,T]$. From Part (2) of Theorem \ref{2VDD-T1.2}, we
   obtain $\rho(\cdot,t)|_{U^t}=0$, $t\in[0,T]$.

From (\ref{2VDD-E2.16-4}), (\ref{2VDD-E2.22}) and
(\ref{2VDD-E2.28}), we have
    $$
    \int^\infty_1\int(|\rho^\delta-\tilde{\rho}|^4+|F^\delta|^4)dxdt\leq
    C.
    $$
From the convergence of $\{\rho^\delta\}$ and $\{F^\delta\}$, we
get
     \begin{equation}
    \int^\infty_1\int(|\rho-\tilde{\rho}|^4+|F|^4)dxdt\leq
    C.\label{2VDD-E4.8}
     \end{equation}
From (\ref{2VDD-E2.25}), we have
    $$
    \int|\rho^\delta-\tilde{\rho}|^4(x,t)dx\leq
    \int|\rho^\delta-\tilde{\rho}|^4(x,s)dx+C\int^{N+1}_N\int|F^\delta|^4dxdt,
    $$
where $s,t\in[N,N+1]$, $N>1$. Integrating it with $s$ in
$[N,N+1]$, we obtain
      $$
    \sup_{t\in[N,N+1]}\int|\rho^\delta-\tilde{\rho}|^4(x,t)dx\leq
    C\int^{N+1}_N\int(|\rho^\delta-\tilde{\rho}|^4+|F^\delta|^4)dxdt.
    $$
From the convergence of $\{\rho^\delta\}$ and $\{F^\delta\}$, we
get
      $$
    \sup_{t\in[N,N+1]}\int|\rho-\tilde{\rho}|^4(x,t)dx\leq
    C\int^{N+1}_N\int(|\rho-\tilde{\rho}|^4+|F|^4)dxdt.
    $$
Letting $N\rightarrow\infty$, using (\ref{2VDD-E4.8}), we can easily
obtain (\ref{2VDD-E1.28})--(\ref{2VDD-E1.29}).
 {\hfill $\square$\medskip}

\section{Propagation of Singularities}\label{2VDD-S5}

Before proving Theorem \ref{2VDD-T1.5}, using the similar method in
\cite{Hoff2008}, we give the following three lemmas.

\begin{lem}\label{2VDD-L5.1}
  Given $x_0\in\mathbb{R}^2$ and $R>0$, there are positive constants $\delta_0$ and $r_0$,
  and a subsequence $\delta\equiv \delta_j\rightarrow 0$, such
  that $X^{\delta}(t,\cdot)^{-1}B_{r_0}(x_0^t)\subset B_R(x_0)$
  for all $\delta\leq \delta_0$.
\end{lem}
\begin{proof}
  By Lemma \ref{2VDD-L4.1}, there is a sequence
  $\delta\rightarrow0$ such that
  $y_0^\delta:=X^\delta(t,\cdot)^{-1}(x_0^t)$ tends to $x_0$ as
  $\delta\rightarrow0$. Therefore, it suffices to show that there
  is a positive constant $r_0$ such that
  $|X^\delta(t,\cdot)^{-1}(x)-y_0^\delta|<R$ for sufficiently
  small $\delta$ and for all $x\in B_{r_0}(x_0^t)$.

  Letting
  $y^\delta=X^\delta(t,\cdot)^{-1}(x)$, from (\ref{2VDD-E4.1}), we have
    $$
        X^\delta(s;x,t)-X^\delta(s;x^t_0,t)
    =x-\int^t_s
    u^\delta(X^\delta(\tau;x,t),\tau)d\tau
    -x^t_0+\int^t_s
    u^\delta(X^\delta(\tau;x^t_0,t),\tau)d\tau,
    $$
for any $s\in[0,t)$. Using a similar argument as that in the proof
(\ref{2VDD-E4.8-0}), we have
    $$
    |y^\delta-y_0^\delta|\leq\exp(1-e^{-C(1+t)})|x_0^t-x|^{\exp(-C(1+t))}<R
    $$
     of sufficiently small radius $r_0$. Then, we
finish the proof.
\end{proof}

\begin{lem}\label{2VDD-L5.2}
  Let $\delta\rightarrow 0$ be the sequence of Lemma
  \ref{2VDD-L5.1}, then for all $r_0>0$, $X^\delta(t,\cdot)^{-1}\rightarrow
  X(t,\cdot)^{-1}$ uniformly on $B_{r_0}(x_0^t)$.
\end{lem}
\begin{proof}
  Let $y^\delta:=X^\delta(t,\cdot)^{-1}(x)$ and
  $y:=X(t,\cdot)^{-1}(x)$ for $x\in B_{r_0}(x_0^t)$. For $0\leq s\leq
  t$, we have
    \begin{eqnarray*}
      &&|X^\delta(s,y^\delta)-X(s,y)|\\
        &=&\left|
      x-\int^t_s u^\delta(X^\delta(\tau,y^\delta),\tau)d\tau
      -      x+\int^t_s u(X(\tau,y),\tau)d\tau
      \right|\\
        &\leq& \int^t_s|u^\delta(X^\delta(\tau,y^\delta),\tau)
        -u^\delta(X(\tau,y),\tau)|d\tau
        +\int^t_s|u(X(\tau,y),\tau)
        -u^\delta(X(\tau,y),\tau)|d\tau\\
            &\leq&\int^t_s g^\delta(\tau)m(|X^\delta(\tau,y^\delta)
        -X(\tau,y)|)d\tau
        +\int^t_s|u(X(\tau,y),\tau)
        -u^\delta(X(\tau,y),\tau)|d\tau.
    \end{eqnarray*}
Using a similar argument as that in the proof (\ref{2VDD-E4.6}), we
obtain
    \begin{eqnarray*}
      &&|y^\delta-y|=|X^\delta(0,y^\delta)-X(0,y)|\\
        &\leq&\exp(1-e^{-C(1+T)})\left(\int^t_s|u(X(\tau,y),\tau)
        -u^\delta(X(\tau,y),\tau)|d\tau\right)^{\exp(-C(1+T))}.
    \end{eqnarray*}
From (\ref{2VDD-E4.4}) and the uniform convergence of $u^\delta$
to $u$ on compact sets in $\mathbb{R}^2\times(0,\infty)$, we can
get that $y^\delta(x)\rightarrow y(x)$ uniformly on
$B_{r_0}(x_0^t)$.
\end{proof}

\begin{lem}\label{2VDD-L5.3}
  Given $r_0$ sufficiently small and given $t>0$, there is a
  nondecreasing function $\eta:[0,\infty)\rightarrow\mathbb{R}$
  satisfying $\lim_{r\rightarrow0}\eta(r)=0$ such that, for
  $\delta$ as in Lemma \ref{2VDD-L5.2} sufficient small,
    $$
    |X^\delta(t_2,y_2)
    -X^\delta(t_2,y_1)|\leq \eta(|X^\delta(t_1,y_2)
    -X^\delta(t_1,y_1)|),
    $$
  where    $t_1,t_2\in[0,t]$ and $X^\delta(t,y_1)$, $X^\delta(t,y_2)\in
  B_{r_0}(x_0^t)$.
\end{lem}
\begin{proof}
  We have that, for $0\leq t_1,t_2\leq t$,
    \begin{eqnarray*}
      &&|X^\delta(t_2,y_2)
    -X^\delta(t_2,y_1)|\\
            &\leq& |X^\delta(t_1,y_2)
    -X^\delta(t_1,y_1)|
    +\left|\int^{t_2}_{t_1}
    |u^\delta(X^\delta(s,y_1^\delta),s)
    -u^\delta(X^\delta(s,y_2^\delta),s)|ds
    \right|\\
         &\leq& |X^\delta(t_1,y_2)
    -X^\delta(t_1,y_1)|
    +\int^{t}_{0}
    g^\delta(s)m(|X^\delta(s,y_1^\delta)
    -X^\delta(s,y_2^\delta)|)ds.
    \end{eqnarray*}
Using a similar argument as that in the proof (\ref{2VDD-E4.6}), we
obtain
    $$
    |X^\delta(t_2,y_2)
    -X^\delta(t_2,y_1)|\leq \exp(1-e^{-C(1+t)})|X^\delta(t_1,y_2)
    -X^\delta(t_1,y_1)|^{\exp(-C(1+t))}.
    $$
This finishes the proof of this lemma.
\end{proof}

\noindent\textbf{Proof of Theorem \ref{2VDD-T1.5}.}

First, we show that
    \begin{equation}
      \mathrm{osc}(\rho;x_0^t,E^t)\leq \mathrm{osc}(\rho;x_0^t,(E\cap
B)^t)=0,\label{2VDD-E5.1}
    \end{equation}
for some open ball $B$ centered at $x_0$, where $(E\cap
B)^t:=X(t,\cdot)(E\cap B)$. In deed,  since the map
$X(t,\cdot)^{-1}: V^t\rightarrow V$ is continuous by Theorem
\ref{2VDD-T1.3} (2), there is a positive constant $r_B$ such that
$B_{r_1}(x_0^t)\subset B^t$ for $0<r_1\leq r_B$. This implies that
$E^t\cap B_{r_1}(x_0^t)\subset (E\cap B)^t$, so that, for any
$R>0$ and $r_1\leq\min\{R,r_B\}$,
    $$
    (\mathrm{esssup}-\mathrm{essinf})\rho(\cdot,t)|_{E^t\cap B_{r_1}(x_0^t)}
    \leq (\mathrm{esssup}-\mathrm{essinf})\rho(\cdot,t)|_{(E\cap B)^t\cap
    B_R(x_0^t)}.
    $$
Letting first $r_1\rightarrow0$ and then $R\rightarrow0$, we can
obtain (\ref{2VDD-E5.1}).

\noindent\textbf{Case 1.} If
$\lim_{R\rightarrow0}\mathrm{essinf}\rho_0|_{E\cap B_R(x_0)}=0$,
then for any $\delta_0>0$, there exists $r_0>0$ such that
$\rho_0|_{E\cap B_{r_0}(x_0)}\leq \delta_0$. Using a similar
argument as that in the proof of Theorem \ref{2VDD-T1.4}, we have
    $$\rho|_{(E\cap B_{r_0}(x_0))^t}\leq C(T)\delta_0,$$
and
    $$\mathrm{osc}(\rho(\cdot,t);x^t_0,(E\cap B_{r_0}(x_0))^t)\leq C(T)\delta_0,
    \ t\in(0,T].$$
From (\ref{2VDD-E5.1}), we have
    $$\mathrm{osc}(\rho(\cdot,t);x^t_0,E^t)\leq C(T)\delta_0,
     \ 0<t\leq T.$$
Letting $\delta_0\rightarrow0$, we get
    $$\mathrm{osc}(\rho(\cdot,t);x^t_0,E^t)=0.$$

\noindent\textbf{Case 2.} If
$\lim_{R\rightarrow0}\mathrm{essinf}\rho_0|_{E\cap B_R(x_0)}>0$,
then there exist $r_0>0$ and $\underline{\rho}>0$ such that
$\rho_0|_{E\cap B_{r_0}(x_0)}\geq \underline{\rho}$.

 Let $B$ be an
open ball centered at $x_0$ such that $\overline{B}\subset
B_{r_0}(x_0)$. By Theorem \ref{2VDD-T1.2} (2) and Lemma
\ref{2VDD-L5.1}, we can choose a positive constant $r_1\in(0,r_0)$
such that, if $x_j=X(t,y_j)=X^\delta(t,y^\delta_j)\in
B_{r_1}(x_0^t)\cap (E\cap B)^t$, $j=1,2$, then $y_j,y_j^\delta\in
B\cap E$. From the mass equation (\ref{2VDD-E1.1}) for
$(\rho^\delta,u^\delta)$, we obtain
   \begin{eqnarray*}
\frac{d}{dt}\Lambda(\rho^\delta(X^\delta(t;\cdot,0),t))|^{y_2}_{y_1}
 & =&-P(\rho^\delta(X^\delta(t;\cdot,0),t))|^{y_2}_{y_1}-
    F^\delta(X^\delta(t;\cdot,0),t)|^{y_2}_{y_1}\\
         & =&-a(t)\Lambda(\rho^\delta(X^\delta(t;\cdot,0),t))|^{y_2}_{y_1}-
    F^\delta(X^\delta(t;\cdot,0),t)|^{y_2}_{y_1},
   \end{eqnarray*}
 where $a(t)=\frac{P(\rho^\delta(X^\delta(t;\cdot,0),t))|^{y_2}_{y_1}}
 {\Lambda(\rho^\delta(X^\delta(t;\cdot,0),t))|^{y_2}_{y_1}}$. From
 (\ref{2VDD-E1.15}) and
 Theorem \ref{2VDD-T1.3}, we have $|a(t)|\leq C$. Using Gronwall's
 inequality, we get
       \begin{eqnarray}
        &&|\Lambda(\rho^\delta(x_2,t))-\Lambda(\rho^\delta(x_1,t))|\label{2VDD-E5.2}\\
 &\leq&C(T)|\Lambda(\rho^\delta_0(y_2^\delta))-\Lambda(\rho^\delta_0(y_1^\delta))|
 +C(T)\int^t_0|
    F^\delta(X^\delta(s,y_2^\delta),s)-F^\delta(X^\delta(s,y_1^\delta),s)|ds\nonumber\\
 &\leq&C(T)|\Lambda(\rho^\delta_0(y_2^\delta))-\Lambda(\rho^\delta_0(y_1^\delta))|
 +C(T)\sup_{0\leq s\leq t}|
    X^\delta(s,y_2^\delta)-X^\delta(s,y_1^\delta)|^\frac{q}{2+q}\int^t_0
    g_{F^\delta}ds,\nonumber
   \end{eqnarray}
where
$g_{F^\delta}(s)=<F^\delta(\cdot,s)>^\frac{q}{2+q}_{\mathbb{R}^2}$.
Using a similar argument as that in Proposition \ref{2VDD-P2.8}, we
have
    \begin{equation}
    \int^T_0 g_{F^\delta}ds\leq C(T).\label{2VDD-E5.3}
    \end{equation}
 From Lemma
\ref{2VDD-L5.3}, we obtain
   \begin{equation}
   \sup_{0\leq s\leq t} |
    X^\delta(s,y_2^\delta)-X^\delta(s,y_1^\delta)|^\frac{q}{2+q}
    \leq (\eta(|x_2-x_1|))^\frac{q}{2+q}\leq (\eta(2r_1))^\frac{q}{2+q}.\label{2VDD-E5.4}
    \end{equation}

Since $y_j\in E\cap B$, there is a $r_2>0$ such that
$B_{r_2}(y_j)\subset E\cap B$. By Lemma \ref{2VDD-L5.2},
$y_j^\delta\in B_{\frac{r_2}{2}}(y_j)$ for $\delta$ sufficiently
small. Thus, if $\delta<\frac{r_2}{2}$, then
$|y-y_j|\leq|y-y_j^\delta| +|y_j^\delta-y_j|\leq
\delta+\frac{r_2}{2}<r_2$ for all $y\in B_\delta(y_j^\delta)$;
that is $B_\delta(y_j^\delta)\subset B_{r_2}(y_j)\subset E\cap B$.
Also, by lemmas \ref{2VDD-L5.2}-\ref{2VDD-L5.3}, for $y\in
B_\delta(y_j^\delta)$,
    $$
        |y-x_0|\leq |y-y_j^\delta|+|y_j^\delta-y_j|+|y_j-x_0|
        \leq \delta+|y_j^\delta-y_j|+\eta(r_1)\leq 2\eta(r_1),
    $$
that is, $B_\delta(y_j^\delta)\subset B_{2\eta(r_1)}(x_0)$ for
$\delta$ sufficiently small. Thus, $B_\delta(y_j^\delta)\subset
(E\cap B)\cap B_{2\eta(r_1)}(x_0)$. Then, since
$\rho_0^\delta(y_j^\delta)=\int_{B_\delta(y_j^\delta)}j_\delta(y_j^\delta-y)\rho_0^\delta
(y)dy+\delta$, we have
      \begin{equation}
    \mathrm{essinf}\rho_0|_{(E\cap B)\cap B_{2\eta(r_1)}(x_0)}
    \leq \rho_0^\delta(y_j^\delta)-\delta\leq
     \mathrm{esssup}\rho_0|_{(E\cap B)\cap B_{2\eta(r_1)}(x_0)}.\label{2VDD-E5.5}
        \end{equation}
From (\ref{2VDD-E5.2})--(\ref{2VDD-E5.5}), we obtain
    $$
    |\Lambda(\rho^\delta(x_2,t))-\Lambda(\rho^\delta(x_1,t))|
    \leq C(T)(\mathrm{esssup}-\mathrm{essinf})\rho_0|_{(E\cap B)\cap B_{2\eta(r_1)}(x_0)}
    +C(T)\eta(r_1)^\frac{q}{2+q},
    $$
for $x_1,x_2\in (E\cap B)^t\cap B_{r_1}(x_0^t)$, sufficiently small
$r_1>0$ and sufficiently small $\delta>0$. Taking the limit as
$\delta\rightarrow0$, we get
    $$
     |\Lambda(\rho(x_2,t))-\Lambda(\rho(x_1,t))|
    \leq C(T)(\mathrm{esssup}-\mathrm{essinf})\rho_0|_{(E\cap B)\cap B_{2\eta(r_1)}(x_0)}
    +C(T)\eta(r_1)^\frac{q}{2+q},
    $$
for $x_1,x_2\in (E\cap B)^t\cap B_{r_1}(x_0^t)$ and sufficiently
small $r_1>0$. Then, we have
    $$
     (\mathrm{esssup}-\mathrm{essinf})\rho|_{(E\cap B)^t\cap B_{r_1}(x_0^t)}
    \leq C(T)(\mathrm{esssup}-\mathrm{essinf})\rho_0|_{(E\cap B)\cap B_{2\eta(r_1)}(x_0)}
    +C(T)\eta(r_1)^\frac{q}{2+q},
    $$
for  sufficiently small $r_1>0$. Letting $r_1\rightarrow0$, using
the condition osc$(\rho_0;x_0,E\cap B)=0$, we obtain
osc$(\rho(\cdot,t);x_0^t,(E\cap B)^t)=0$. From (\ref{2VDD-E5.1}), we
complete the proof of Theorem \ref{2VDD-T1.5}.
 {\hfill
$\square$\medskip}

\noindent\textbf{Proof of Theorem \ref{2VDD-T1.6}.}

It follows immediately from Theorem \ref{2VDD-T1.5} that, under the
conditions of Theorem \ref{2VDD-T1.6} (a), $\rho(\cdot,t)$ has a
one-sided limit with respect to
$\mathcal{M}^t:=X(t,\cdot)\mathcal{M}$ at the point $X(t,x_0)$.
Then, since $F=(\lambda+\mu)\mathrm{div}u-P(\rho)+P(\tilde{\rho})$
is H\"{o}lder continuous for $t>0$ by (\ref{2VDD-E1.22}),
$\mathrm{div}u(\cdot,t)$ has a one-sided limit  with respect to
$\mathcal{M}^t:=X(t,\cdot)\mathcal{M}$ at the point $X(t,x_0)$. If
these limits exist from both sides of $\mathcal{M}$ at $x_0$, using
the H\"{o}lder continuity of $F$, then the Rankine-Hugoniot
condition (\ref{2VDD-E1.31}) holds at the point $X(t,x_0)$.

To prove the regularity assertions in Theorem \ref{2VDD-T1.6} (a),
we consider the following two case.

\textbf{Case 1.} If $\rho_0(x_0+)=0$, then using a similar argument
as that in the proof of Theorem \ref{2VDD-T1.4}, we have
$\rho(X(t,x_0)+,t)=0$ for each $t>0$. Thus, the maps  $t\mapsto
\rho(X(t,x_0)+,t)$ and $t\mapsto P(\rho(X(t,x_0)+,t))$ are
  in $C^1([0,\infty))$.

\textbf{Case 2.} If $\rho_0(x_0+)>0$, then there exist
$\underline{\rho}>0$ and  $r_0>0$, such that $\rho_0|_{E_+\cap
B_{r_0}(x_0)}\geq \underline{\rho}$. Denote
$\Lambda^\delta:=\Lambda(\rho^\delta)$ and
$\Lambda:=\Lambda(\rho)$. From (\ref{2VDD-E1.18}), we have
    $$
    \Lambda^\delta_t+u^\delta\cdot\nabla \Lambda^\delta
    =-F^\delta-P(\rho^\delta)+P(\tilde{\rho}).
    $$
We will choose a sequence of smooth test functions
$\{\phi^{\delta,h}\}_{\delta,h>0}$  satisfying the equation
$\phi^{\delta,h}_t+\mathrm{div}(\phi^{\delta,h}u^\delta)=0$, so
that
    \begin{equation}
     \left. \int\phi^{\delta,h}\Lambda^\delta dx\right|^t_0
      =-\int^t_0\int \phi^{\delta,h}(F^\delta+P(\rho^\delta)-P(\tilde{\rho}))
      dxds.\label{2VDD-E5.6}
    \end{equation}
To construct $\phi^{\delta,h}$, we let $\{x_h\}_{h>0}$ be a
sequence in $E_+$ and $\{r_h\}$ a sequence of positive numbers
such that $x_h\rightarrow x_0$ and $r_h\rightarrow0$ as
$h\rightarrow0$, and $B_{2r_h}(x_h)\subset E_+\cap B_{r_0}(x_0)$.
Then, define $\phi^{\delta,h}$ the solution of the equation
    $$
    \left\{\begin{array}{l}
        \phi^{\delta,h}_t+\mathrm{div}(\phi^{\delta,h}u^\delta)=0,\\
            \phi^{\delta,h}|_{t=0}=\phi^h_0,
    \end{array}
    \right.
    $$
where $\phi^h_0$ is a smooth function with support in
$B_{r_h}(x_h)$, $\int \phi^h_0dx=1$ and $0\leq \phi^h_0\leq C^h$. It
follows that $\phi^{\delta,h}$ has support in
$X^\delta(t,\cdot)B_{r_h}(x_h)$, $\int\phi^{\delta,h}(x,t)dx=1$ for
$t>0$, and $0\leq \phi^{\delta,h}\leq C^h(T)$ for $0\leq t\leq T$.
This last assertion is a consequence of (\ref{2VDD-E2.72}).

In the following, we will take limits as $\delta$ and $h$ go to
zero in (\ref{2VDD-E5.6}). Form Lemma \ref{2VDD-L5.2}, we have
$X^\delta$ converges uniformly to $X$ on $[0,t]\times
B_{r_0}(x_0)$ as $\delta\rightarrow0$. Then, we can easily prove
that, for each $h>0$, there is a $\delta_0(h)>0$ such that
    \begin{equation}
      X^\delta(s,\cdot)B_{r_h}(x_h)\subset X(s,\cdot)
      B_{2r_h}(x_h),
      \ 0<\delta<\delta_0(h).\label{2VDD-E5.7}
    \end{equation}

We now obtain uniform bounds of the three terms in
(\ref{2VDD-E5.6}). For fixed $t>0$, we get
    $$
    \int\phi^{\delta,h}\Lambda^\delta dx
    =\int\phi^{\delta,h}(\Lambda^\delta-\Lambda)dx+\int\phi^{\delta,h}\Lambda dx
    :=I+II.
    $$
From the convergence of $\rho^\delta$ and (\ref{2VDD-E1.15}), we
obtain that $I\rightarrow0$ as $\delta\rightarrow0$, and $I$ is
bounded by $C^h(T)$. Also, from (\ref{2VDD-E5.7}), we have
    $$
    \mathrm{essinf}
    \Lambda(\cdot,t)|_{X(t,\cdot)B_{2r_h}(x_h)}\leq II\leq \mathrm{esssup}
    \Lambda(\cdot,t)|_{X(t,\cdot)B_{2r_h}(x_h)},
    $$
 for sufficiently
    small $\delta$. Thus, there exists $\delta_0(h)>0$
        \begin{equation}
           \mathrm{essinf}
    \Lambda(\cdot,t)|_{X(t,\cdot)B_{2r_h}(x_h)}-h\leq \int\phi^{\delta,h}\Lambda^\delta dx\leq \mathrm{esssup}
    \Lambda(\cdot,t)|_{X(t,\cdot)B_{2r_h}(x_h)}+h,\label{2VDD-E5.8}
        \end{equation}
 for  $\delta\leq \delta_0(h)$.

 Let
 $\Lambda(x_0^t+,t):=\Lambda(\rho(x_0^t+,t))$. Then,
    $$
    \Lambda(x_0^t+,t)=\lim_{r'\rightarrow0}\mathrm{esssup}
    \Lambda(\cdot,t)|_{B_{r'}(x_0^t)\cap E^t_+}
    =\lim_{r'\rightarrow0}\mathrm{essinf}
    \Lambda(\cdot,t)|_{B_{r'}(x_0^t)\cap E^t_+}.
    $$
For each $r'>0$, we can choose $h_{r'}>0$ such that
$X(t,\cdot)B_{2r_h}(x_h)\subset B_{r'}(x_0^t)\cap E^t_+$ for all
$h\leq h_{r'}$. Thus, for each $r'>0$ and all $h\leq h_{r'}$,
    $$
    \mathrm{essinf}\Lambda(\cdot,t)|_{B_{r'}(x_0^t)\cap E^t_+}
    \leq \mathrm{essinf} \Lambda(\cdot,t)|_{X(t,\cdot)B_{2r_h}(x_h)}
    \leq     \mathrm{esssup}\Lambda(\cdot,t)|_{B_{r'}(x_0^t)\cap
    E^t_+}.
    $$
Taking first the liminf and limsup as $h\rightarrow0$ and then the
limit as $r'\rightarrow0$, we obtain that
    $$
    \lim_{h\rightarrow0}\mathrm{essinf}
    \Lambda(\cdot,t)|_{X(t,\cdot)B_{2r_h}(x_h)}=\Lambda(x_0^t+,t).
    $$
Similarly, we have
     $$
    \lim_{h\rightarrow0}\mathrm{esssup}
    \Lambda(\cdot,t)|_{X(t,\cdot)B_{2r_h}(x_h)}=\Lambda(x_0^t+,t).
    $$
From (\ref{2VDD-E5.8}), we have
    \begin{equation}
      \lim_{h\rightarrow0}\lim_{\delta\rightarrow0}\int(\phi^{\delta,h}\Lambda^\delta)(x,t)dx=\Lambda(x_0^t+,t).
      \label{2VDD-E5.11}
    \end{equation}
Similarly, we obtain
    \begin{equation}
      \lim_{h\rightarrow0}\lim_{\delta\rightarrow0}\int(\phi^{\delta,h}\Lambda^\delta)(x,0)dx=\Lambda(x_0+,0),
    \end{equation}
    \begin{equation}
      \lim_{h\rightarrow0}\lim_{\delta\rightarrow0}
      \int^t_0\int(\phi^{\delta,h}F^\delta)(x,s)dxds=\int^t_0F(x_0^s+,s)ds,
    \end{equation}
    \begin{equation}
      \lim_{h\rightarrow0}\lim_{\delta\rightarrow0}
      \int^t_0\int\phi^{\delta,h}\{P(\rho^\delta)-P(\tilde{\rho})\}dxds=\int^t_0
      \{P(\rho(x_0^s+,s))-P(\tilde{\rho})\}ds.\label{2VDD-E5.14}
    \end{equation}
From (\ref{2VDD-E5.6}) and (\ref{2VDD-E5.11})--(\ref{2VDD-E5.14}),
we have
    \begin{equation}
      \Lambda(x_0^t+,t)- \Lambda(x_0+,0)
      =-\int^t_0\{F(x_0^s+,s)+P(\rho(x_0^s+,s))-P(\tilde{\rho})\}ds.\label{2VDD-E5.13-1}
    \end{equation}
Since $F$ is locally H\"{o}lder continuous in
  $\mathbb{R}^2\times(0,\infty)$, standard ODE theory implies that the map $t\mapsto \Lambda(x_0^t+,t)$ is
  in $C([0,\infty))\cap C^1((0,\infty))$.
From (\ref{2VDD-E1.15}), (\ref{2VDD-E1.15-2}), (\ref{2VDD-E2.66})
and (\ref{2VDD-E2.72-1}), we have
    $$
      \int^t_0\|F(\cdot,s)\|_{L^\infty}ds\leq
      C\int^t_0(\sigma^{-2-\frac{q}{2}})^{\frac{1}{2+2q}}\leq
      Ct^{\frac{3q}{4+4q}}.
    $$
Thus, map $t\mapsto \Lambda(x_0^t+,t)$ is
  in $C^{\frac{3q}{4+4q}}([0,\infty))\cap C^1((0,\infty))$.
   From
 (\ref{2VDD-E1.15}) and
 Theorem \ref{2VDD-T1.3}, the maps  $t\mapsto
\rho(X(t,x_0)+,t)$ and $t\mapsto P(\rho(X(t,x_0)+,t))$ are
  in $C^{\frac{3q}{4+4q}}([0,\infty))\cap C^1((0,\infty))$.

   Since $F$ is locally H\"{o}lder continuous in
  $\mathbb{R}^2\times(0,\infty)$, then the map $t\mapsto \mathrm{div}u(X(t,x_0)+,t)$ is
  locally H\"{o}lder continuous on $(0,\infty)$. These finish the
  proof of Theorem \ref{2VDD-T1.6} (a).

Then, we will prove Theorem \ref{2VDD-T1.6} (c) as follows. From
(\ref{2VDD-E5.13-1}), we have
    $$
    [\Lambda(x_0^t,t)]- [\Lambda(x_0,0)]
      =-\int^t_0[P(\rho(x_0^s+,s))]ds=-\int^t_0a(s,x_0)[\Lambda(x_0^s+,s)]ds.
    $$
Using  Gronwall's inequality, we can finish the proof of Theorem
\ref{2VDD-T1.6}.
 {\hfill
$\square$\medskip}

\section{Nonphysical solution}

\textbf{Proof of Theorem \ref{2VDD-T1.7}.}

First, from (\ref{2VDD-E1.4}), we have
    $$
   \partial_r( \partial_r v+\frac{1}{r}v)=0, \ \textrm{ in }\mathcal{D}'(U),
    $$
where $U=\{(r,t)\in \mathbb{R}^+\times\mathbb{R}^+|
r\in(a(t),b(t))\}$. From the regularity (\ref{2VDD-E1.15-1}), we
have
    \begin{equation}
      v(r,t)=\alpha(t)r+\frac{\beta(t)}{r},\ \textrm{ in
      }\mathcal{D}'(U),\label{2VDD-E6.1}
    \end{equation}
where $\alpha,\beta\in L^1_{loc}([0,\infty))$.

    Then, from (\ref{2VDD-E1.30-1}) and the regularity of the solution in Theorem \ref{2VDD-T1.1}, we
    have
        $$
    \frac{d}{dt}E(t)=\left.\left\{
    vr\left((\lambda+2\mu)(v_r+\frac{v}{r})-P(\varrho)
    \right)    \right\}\right|_{r=a(t)-0},
        $$
where
    $$
    E(t)=\int^{a(t)}_{0}\left(\frac{1}{2}\varrho v^2+\overline{G}(\varrho)
    \right)rdr+\int^t_0\int_0^{a(s)}(\lambda+2\mu)(v_r+\frac{v}{r})^2rdrds.
    $$
     From (\ref{2VDD-E1.15-1}) and (\ref{2VDD-E1.31}), we get
        $$
    [v]=0,\     [(\lambda(\varrho)+2\mu)(v_r+\frac{v}{r})-P(\varrho)]=0,
    \ \textrm{ at }\ (a(t),t),
    $$
 and since $\varrho(a(t)+0,t)=0$,
we obtain
    \begin{equation}
    \frac{d}{dt}E(t)=a(t)v(a(t),t)
    (\lambda(0)+2\mu)(v_r+\frac{v}{r})(a(t)+0,t)
    .\label{2VDD-E6.2}
    \end{equation}
From (\ref{2VDD-E6.1}), we have
    \begin{equation}
    (v_r+\frac{v}{r})(a(t)+0,t)=2\frac{a(t)v(a(t),t)-b(t)v(b(t),t)}{a^2(t)-b^2(t)}.\label{2VDD-E6.3}
    \end{equation}
From (\ref{2VDD-E6.2})--(\ref{2VDD-E6.3}), we can immediately obtain
(\ref{2VDD-E1.33}).
 {\hfill
$\square$\medskip}

\section{Non-global Existence of Regular Solutions}\label{2VDD-S7}
\textbf{Proof of Theorem \ref{2VDD-T1.8}.}

Since the support of the initial spherical symmetric density
  $\rho_0$ is compact, we can assume that
    $$
  \mathrm{supp}\rho_0=\Omega(0)=\{x\in\mathbb{R}^2| |x|\leq R_0\},
    $$
  for some $R_0>0$. We set
    $$
        \Omega(t)=\{x=X(t,\alpha)|\alpha\in \Omega(0)\},
        \ t\in[0,T],
    $$
  where $X(t,\alpha)$ is defined in Theorem \ref{2VDD-T1.2}. From
  the transport equation (\ref{2VDD-E1.1})$_1$, one can easily
  show that supp$\rho(x,t)=\Omega(t)=\{x\in\mathbb{R}^2| |x|\leq
  R(t)\}$. From (\ref{2VDD-E6.1}), we have
    $$
    v(r,t)=\alpha(t)r+\frac{\beta(t)}{r},\ \textrm{ in
      }\mathcal{D}'(U),
    $$
where $u(x,t)=v(r,t)\frac{x}{r}$, $|x|=r$,
$U=\{x\in\Omega^c(t),t\in[0,T]\}$. Since $u\in C^1([0,T];H^k)$, we
have $u(x,t)\equiv0$ in $x\in \Omega^c(t)$, and
$\Omega(t)=\Omega(0)$ for all $0<t<T$.

Now, we introduce the following functional as in \cite{xin,yang3}:
    \begin{eqnarray}
      H(t)&=&\int(x-(1+t)u)^2\rho
      dx+\frac{2}{\gamma-1}(1+t)^2\int A\rho^\gamma dx\label{2VDD-E7.1}\\
        &=&\int x^2\rho dx-2(1+t)\int \rho u\cdot xdx
        +(1+t)^2\int\left(\rho u^2+\frac{2A}{\gamma-1}\rho^\gamma
        \right)dx,\ t\in[0,T].\nonumber
    \end{eqnarray}
Using (\ref{2VDD-E1.29-1})--(\ref{2VDD-E1.30-1}), the Cauchy-Schwarz
inequality and H\"{o}lder's inequality, we obtain
    \begin{eqnarray}
      H'(t)&=&\frac{4(2-\gamma)(1+t)}{\gamma-1}\int_0^\infty A\varrho^\gamma
      r dr+4(1+t)\int_0^\infty
      \lambda(v_r+\frac{v}{r})rdr\nonumber\\
        &&-2(1+t)^2\int_0^\infty(\lambda+2\mu)(v_r+\frac{v}{r})^2rdr\nonumber\\
            &\leq&\frac{4(2-\gamma)(1+t)}{\gamma-1}\int_0^\infty A\varrho^\gamma
      r dr+2\int_0^\infty
      c\varrho^\beta rdr\nonumber\\
        &\leq&\frac{4(2-\gamma)(1+t)}{\gamma-1}\int_0^\infty A\varrho^\gamma
      r dr+\frac{2c\beta}{\gamma}\int_0^\infty
      \varrho^\gamma rdr+\frac{2c(\gamma-\beta)}{\gamma}\Omega(0),\label{2VDD-E7.2}
    \end{eqnarray}
where $t\in[0,T]$. From (\ref{2VDD-E7.1})--(\ref{2VDD-E7.2}), we get
    $$
    H'(t)\leq
    \frac{2(2-\gamma)}{1+t}H(t)+\frac{2c\beta(\gamma-1)}{A\gamma(1+t)^2}H(t)
    +\frac{2c(\gamma-\beta)}{\gamma}\Omega(0)
    $$
and
    \begin{eqnarray*}
    H(t)&\leq&
    (1+t)^{4-2\gamma}e^{-\frac{2c\beta(\gamma-1)}{A\gamma(1+t)}}\left(
    H(0)+\frac{2c(\gamma-\beta)}{\gamma}\Omega(0)\int^t_0
    (1+s)^{2\gamma-4}e^{\frac{2c\beta(\gamma-1)}{A\gamma(1+s)}}ds
    \right)\nonumber\\
    &\leq&
    (1+t)^{4-2\gamma}
    H(0)+\frac{2c(\gamma-\beta)}{\gamma}\Omega(0)
    e^{\frac{2c\beta(\gamma-1)}{A\gamma}}F(t),
    \end{eqnarray*}
where $t\in[0,T]$ and
    $$
    F(t)=\left\{
    \begin{array}{ll}
    \frac{1+t}{2\gamma-3}(1-(1+t)^{3-2\gamma}),&2\gamma-3\neq0,\\
    (1+t)^{4-2\gamma}\ln(1+t),&2\gamma-3=0.
    \end{array}
    \right.
    $$
    From (\ref{2VDD-E7.1}), we have
    $$
    \int A\rho^\gamma dx \leq
    \frac{\gamma-1}{2}(1+t)^{2-2\gamma}
    H(0)+ \frac{c(\gamma-1)(\gamma-\beta)}{\gamma}\Omega(0)
    e^{\frac{2c\beta(\gamma-1)}{A\gamma}}F(t)(1+t)^{-2}:=G(t).
    $$
By the conservation of mass and H\"{o}lder's inequality, we obtain
    $$
        \int\rho_0dx=\int \rho dx
        \leq\left(\int\rho^\gamma dx
        \right)^\frac{1}{\gamma} (\Omega(0))^{\frac{\gamma-1}{\gamma}}
        \leq \left(\frac{G(t)}{A}
        \right)^\frac{1}{\gamma}
        (\Omega(0))^{\frac{\gamma-1}{\gamma}},
        \ t\in[0,T].
    $$
Since $\lim_{t\rightarrow\infty}G(t)=0$, the above inequality
implies that $T$ must be finite.
 {\hfill
$\square$\medskip}

\end{document}